\documentclass[11pt,letterpaper]{amsart}

\usepackage{amsmath}
\usepackage{amsthm}
\usepackage{hyperref}
\usepackage{amsfonts,graphics,amsthm,amsfonts,amscd,latexsym}
\usepackage{epsfig}
\usepackage{flafter}
\usepackage{mathtools}
\usepackage{comment}
\usepackage{stmaryrd}
\usepackage{tikz}
\usepackage{tikz-cd}

\usepackage{mathabx,epsfig}

\hypersetup{
    colorlinks=true,    
    linkcolor=blue,          
    citecolor=blue,      
    filecolor=blue,      
    urlcolor=blue           
}
\usepackage{tikz}
\usetikzlibrary{graphs,positioning,arrows,shapes.misc,decorations.pathmorphing}

\tikzset{
    >=stealth,
    every picture/.style={thick},
    graphs/every graph/.style={empty nodes},
}

\tikzstyle{vertex}=[
    draw,
    circle,
    fill=black,
    inner sep=1pt,
    minimum width=5pt,
]
\usepackage[position=top]{subfig}
\usepackage{amssymb}
\usepackage{color}

\calclayout

\usetikzlibrary{decorations.pathmorphing}
\tikzstyle{printersafe}=[decoration={snake,amplitude=0pt}]

\newcommand{\Pic}{\operatorname{Pic}}

\newcommand{\Diff}{\operatorname{Diff}}

\newcommand{\codim}{\operatorname{codim}}

\newcommand{\Spec}{\operatorname{Spec}}

\newcommand{\lcm}{\operatorname{lcm}}

\newcommand{\Aut}{\operatorname{Aut}}

\newcommand{\pp}{\mathbb{P}}

\newcommand{\qq}{\mathbb{Q}}
\newcommand{\zz}{\mathbb{Z}}
\newcommand{\nn}{\mathbb{N}}

\newcommand{\kk}{\mathbb{K}}

\newcommand{\Af}{\mathbb{A}}

\def\O#1.{\mathcal {O}_{#1}}			
\def\pr #1.{\mathbb P^{#1}}				
\def\af #1.{\mathbb A^{#1}}			
\def\ses#1.#2.#3.{0\to #1\to #2\to #3 \to 0}	
\def\xrar#1.{\xrightarrow{#1}}			
\def\K#1.{K_{#1}}						
\def\bA#1.{\mathbf{A}_{#1}}			
\def\bM#1.{\mathbf{M}_{#1}}				
\def\bN#1.{\mathbf{N}_{#1}}
\def\bL#1.{\mathbf{L}_{#1}}				
\def\bB#1.{\mathbf{B}_{#1}}				
\def\bK#1.{\mathbf{K}_{#1}}			
\def\subs#1.{_{#1}}					
\def\sups#1.{^{#1}}						

\DeclareMathOperator{\coeff}{coeff}	

\DeclareMathOperator{\Supp}{Supp}

\newcommand{\rar}{\rightarrow}

\usepackage{tikz}
\usetikzlibrary{matrix,arrows,decorations.pathmorphing}

  \newtheorem{introthm}{Theorem}

  \newtheorem{introcor}[introthm]{Corollary}
  \newtheorem{introprop}[introthm]{Proposition}
  
  \newtheorem{theorem}{Theorem}[section]
  \newtheorem{lemma}[theorem]{Lemma}
  
  \newtheorem{proposition}[theorem]{Proposition}
  \newtheorem{corollary}[theorem]{Corollary}

\theoremstyle{definition}
  \newtheorem{notation}[theorem]{Notation}
  \newtheorem{definition}[theorem]{Definition}
  \newtheorem{example}[theorem]{Example}

\newtheorem{remark}[theorem]{Remark}

\theoremstyle{remark}

\numberwithin{equation}{section}

\usepackage[all]{xy}

\usepackage{soul}

\usepackage{mathrsfs}

\begin{document}

\title[Index of coregularity zero log Calabi--Yau pairs]{Index of coregularity zero log Calabi--Yau pairs}

\author[S.~Filipazzi]{Stefano Filipazzi}
\address{
EPFL SB MATH CAG
MA C3 625 (B\^atiment MA)
Station 8
CH-1015 Lausanne, Switzerland.
}
\email{stefano.filipazzi@epfl.ch}

\author[M.~Mauri]{Mirko Mauri}
\address{Institute of Science and Technology Austria, 3400 Klosterneuburg, Austria.}
\email{mirko.mauri@ist.ac.at}

\author[J.~Moraga]{Joaqu\'in Moraga}
\address{UCLA Mathematics Department, Box 951555, Los Angeles, CA 90095-1555, USA.
}
\email{jmoraga@math.ucla.edu}

\thanks{
SF was partially supported by ERC starting grant \#804334. MM was supported by the University of Michigan and the Institute of Science and Technology Austria. This project has received funding from the European Union’s Horizon 2020 research and innovation
programme under the Marie Skłodowska-Curie grant agreement No 101034413.
}

\begin{abstract}
In this article, we study the index of log Calabi--Yau pairs $(X,B)$ of coregularity 0. We show that $2\lambda(K_X+B)\sim 0$, where
$\lambda$ is the Weil index of $(X,B)$. This is in contrast to the case of klt Calabi--Yau varieties, where the index can grow doubly exponentially with the dimension. Our sharp bound on the index extends to the context of generalized log Calabi--Yau pairs, semi-log canonical pairs, and isolated log canonical singularities of coregularity 0.
As a consequence, we show that the index of a variety appearing in the Gross--Siebert program or in the Kontsevich--Soibelman program is at most $2$.
Finally, we discuss applications to Calabi--Yau varieties endowed with a finite group action, including holomorphic symplectic varieties endowed with a purely non-symplectic automorphism.
\end{abstract}

\maketitle
\setcounter{tocdepth}{1} 
\tableofcontents

\section{Introduction}

A \emph{log Calabi--Yau pair of coregularity $0$} is a pair $(X, B)$ consisting of a proper
normal variety $X$ and an effective $\qq$-divisor $B \sim_{\qq} - K_X$ such that $(X,B)$ is log canonical and whose dual complex $\mathcal{DMR}(X,B)$ (see Definition \ref{rmk:dualcomplex}) has maximal dimension $\dim \mathcal{DMR}(X,B) = \dim X -1$.
For instance, a normal toric variety with its toric boundary is an example of a log Calabi--Yau pair of coregularity $0$. 

In this article, we study the index of log Calabi--Yau pairs of coregularity $0$.

\subsection{Index conjecture}
From the point of view of the Minimal Model Program (MMP), Calabi--Yau varieties are fundamental building blocks of algebraic varieties.
Their canonical line bundle 
is numerically trivial, i.e., 
its restriction to any curve has degree 0. 
By the abundance conjecture,
which is known in this special case \cite{Gon13}, the canonical divisor of a Calabi--Yau variety
is $\qq$-linearly trivial, i.e., it is a torsion
element of the Picard group.
The smallest positive integer $m$ for which 
$mK_X \sim 0$ is known as the (Cartier) {\em index} of the Calabi--Yau variety.
For instance, $m=1$ if $\dim X =1$, and $m$ divides $12$ if $\dim X =2$ and $X$ is smooth.

Inspired by the MMP, it is natural to study the same question 
for pairs $(X,B)$ with mild singularities. 
Here, mild singularities mean either
klt or log canonical.
When we consider mildly singular varieties
or pairs, the situation gets more interesting.
In this case, the index will also depend
on the coefficient set of $B$.
For instance, 
if we consider $2$-dimensional klt Calabi--Yau
pairs $(X,B)$ with standard coefficients, the largest possible index
of $K_X+B$ is 66;
see~\cite[Proposition 4.9]{Ish00} for the bound and ~\cite[Main Theorem and \S3]{Kondo92} for its sharpness.
It is conjectured that, once we fix
the dimension and coefficient set, 
there are only finitely many possible indices
for the divisor $K_X+B$.
This conjecture is known as the {\em index conjecture}
in birational geometry.
It is known up to dimension 3
due to the work of Jiang~\cite{Jiang21} and Xu~\cite{Xu20}.
In~\cite{ETW22}, the authors construct a sequence of klt Calabi--Yau varieties $X_d$ for which the index grows doubly exponentially with the dimension $d$. 
These Calabi--Yau varieties, regarded as log Calabi--Yau pairs with empty boundary, have coregularity equal to their dimension.
In contrast, our results show that, if the coregularity is minimal (i.e., equal to 0), the index does not depend on the dimension.

\subsection{Main Theorem}\label{sec:MainThm}
In this article, we focus on log Calabi--Yau pairs $(X,B)$ of coregularity 0; see Definition~\ref{defn:coreg}. 
These pairs arise naturally in mirror symmetry, and they are often referred to as pairs of maximal intersection; see~\cite{KX16}. Indeed, the condition that the dual complex has maximal dimension means that, up to a dlt modification, the maximal number of irreducible components of $B^{=1}$ intersecting at some point is the maximal possible, namely $\dim X$.
These pairs also appear in relation to maximally unipotent degenerations of Calabi--Yau varieties; see \S\ref{sec:applicationmirrorsymmetry} and also~\cite{KX16, NX16, NXYu19}. In this setting, we prove
that the index of $K_X+B$ is independent
of the dimension and only depends 
on the coefficient set of $B$, and it is actually $1$ or $2$ if $B$ is reduced.

\begin{introthm}\label{introthm:main-thm}
Let $(X,B,\bM.)$ be a projective generalized log Calabi--Yau pair of coregularity 0 and Weil index $\lambda$.
Then, we have that $\lambda'(K_X+B+\bM X.)\sim 0$,
where $\lambda'=\lcm(\lambda,2)$.
\end{introthm}

See Definition~\ref{def:Weilindex}
for the Weil index of a generalized pair.
Note that in order to have the linear equivalence
$D\sim 0$, the divisor $D$ must be Weil.
Hence, multiplying $K_X+B+\bM X.$ by $\lambda$ is necessary to compute its Cartier index.
In Example \ref{ex:RPN}, we show that taking the least common multiple of $\lambda$ and $2$ is needed in general.
The following special case of our main result is a key step in the proof of Theorem \ref{introthm:main-thm}.

\begin{introprop}
\label{cor:reducedCY}\label{dlt:reduced}
Let $(X,B)$ be a proper log Calabi--Yau pair of coregularity 0 with $B=B\sups =1.$.
Then, we have that $2 (K_X+B)\sim 0$.
\end{introprop}

In many circumstances, it is important to deal with pairs with coefficients greater than $\frac{1}{2}$.
For instance, in moduli theory, this implies the flatness of the boundary divisors; see \cite[Theorem 2.82]{Kollar21}.
In the context of log Calabi--Yau pairs of coregularity 0 and large boundary, we obtain the following corollary.

\begin{introcor}\label{cor:frac1/2}
Let $(X,B)$ be a projective log Calabi--Yau pair of coregularity $0$.
Assume that all the coefficients of $B$ are greater than or equal to $\frac{1}{2}$.
Then, the divisor $2B$ is reduced and 
$2(K_X+B)\sim 0$.
In particular, if the coefficients of $B$ are strictly greater than $\frac{1}{2}$, then $B$ is integral and $2(K_X+B)\sim 0$.
\end{introcor}

\subsection{Orientability of dual complexes}\label{sec:orientability:intro}
The dual complex $\mathcal{D}(B)$ of a dlt pair $(X,B)$ is a simplicial object encoding the combinatorial information about multiple intersections of irreducible components of $B$; see Definition \ref{def:dualcomplex}.
It has been clear since \cite[\S17]{KX16} that the index of $(X,B)$ is a measure of the orientability of $\mathcal{D}(B)$.
What Corollary \ref{cor_orientability} says is that the index of $(X,B)$ does not contain any information but the orientability of $\mathcal{D}(B)$.

\begin{introcor}
\label{cor_orientability}
Let $(X,B)$ be a proper dlt log Calabi--Yau pair of coregularity 0 and dimension $n+1$ with $B=B\sups =1.$.
Then, there exists an isomorphism of singular and coherent cohomology groups
\[H^{n}(\mathcal{D}(B), \kk) \simeq H^0(X, \mathcal{O}_{X}(K_X + B)),\]
and the following facts hold:
\begin{enumerate}
    \item[(i)]  $K_{X} + B \sim 0$ if and only if $\mathcal{D}(B)$ is an orientable closed pseudo-manifold; and
    \item[(ii)] if $\mathcal{D}(B)$ is not an orientable closed pseudo-manifold, there exists a quasi-\'{e}tale  double cover $(\widetilde{X}, \widetilde{B}) \to (X, B)$ such that $\mathcal{D}(\widetilde{B})$ is an orientable closed pseudo-manifold and $\mathcal{D}(B) \simeq  \mathcal{D}(\widetilde{B})/(\zz/2\zz)$.
\end{enumerate}
\end{introcor}

We refer the reader to Definition~\ref{defn:pseudo}
and Definition~\ref{defn:orient} for the orientability of pseudo-manifolds.
In the previous corollary, quasi-\'etale means \'etale in codimension $1$.
Note that the quotient map $\mathcal{D}(\tilde{B})\rightarrow \mathcal{D}(B)$ in Corollary \ref{cor_orientability}
is not necessarily a topological covering space.
In fact, closed pseudo-manifolds admit no orientable topological covering spaces in general, and branched orientation covers are the best we can expect; see Example~\ref{ex:pseudo-not-double-cover} and~\cite[\S5.2]{Matthews}.

\subsection{Mirror symmetry}\label{sec:ms}
Mirror symmetry predicts a relation between pairs of Calabi--Yau varieties.
The goal is to develop a general mechanism to produce mirror pairs and describe
how this duality exchanges geometric structures through the two sides of the mirror.
Various approaches have been developed to achieve this purpose, e.g., the SYZ conjecture \cite{SYZ01}, the Gross--Siebert program \cite{Gross13}, and the Kontsevich--Soibelman program \cite{KS06}.
In all these programs, one tries to degenerate the Calabi--Yau variety to a simple normal crossing Calabi--Yau variety
of coregularity $0$, also known as \emph{large complex limit} or \emph{maximal degeneration}. 
Here, we show that the Calabi--Yau varieties $X$ 
that appear in a maximal degeneration have index at most $2$, namely $2K_{X} \sim 0$.\footnote{Here, a Calabi--Yau variety that appears in a maximal degeneration satisfies $K_{X} \sim_{\qq}0$, while some authors require $K_{X} \sim 0$. Our theorems show that $K_{X} \sim 0$ or $2K_{X} \sim 0$ are automatic.}

\begin{introthm}\label{introcor-GS}
Let $\pi \colon \mathcal{X} \rar C$ be a minimal family of Calabi--Yau varieties of coregularity $0$ over $c_0 \in C$.
Let $U$ be the maximal open subset over which the family is locally stable.
Then, for every $c\in U$, we have that
$2K_{\mathcal{X}_c}\sim 0$.
\end{introthm}
See Definition~\ref{defn:minimalmaximal} for the notion of a minimal family of Calabi--Yau varieties of coregularity $0$ over a point, and Definition~\ref{def:locallystable} 
for the notion of local stability of a family.
Theorem \ref{introcor-GS} is a special case of a more general statement regarding Calabi--Yau pairs; see Theorem \ref{introcor-GS-pairs}.

\begin{introcor}\label{cor:nonorientamirror}
Let $\pi \colon \mathcal{X} \rar C$ be a minimal family of Calabi--Yau varieties of coregularity $0$ over $c_0 \in C$ of relative dimension $n$. Suppose that the pair $(\mathcal{X}, \mathcal{X}_{c_0})$ is dlt. 
Then we have
\begin{equation}\label{eq:identidualcomplex}
    H^{n}(\mathcal{D}(\mathcal{X}_{c_0}), \kk) \simeq H^0(\mathcal{X}_{c}, \mathcal{O}_{\mathcal{X}_{c}}(K_{\mathcal{X}_{c}})).
\end{equation}
In particular, if $K_{\mathcal{X}_c} \not\sim 0$, then $\mathcal{D}( \mathcal{X}_{c_0})$ is not orientable.
\end{introcor}

We can improve the estimate on the index in Theorem~\ref{introcor-GS} by imposing additional hypotheses on the special fiber, which are natural, especially from the point of view of the Gross--Siebert program.

\begin{introthm}\label{introcor-GS2}
Let $\pi \colon \mathcal{X} \rar C$ be a minimal family of Calabi--Yau varieties of coregularity $0$ over $c_0 \in C$. Assume that
\begin{itemize}
    \item the fiber $\mathcal{X}_{c_0}$ is reduced and toric simple normal crossing, i.e., a union of smooth toric varieties intersecting along toric strata; and
    \item the dual complex $\mathcal{D}(\mathcal{X}_{c_0})$ is simply connected.
\end{itemize}
Let $U$ be the maximal open subset over which the family is locally stable.
Then, for every $c\in U$, we have that
$K_{\mathcal{X}_c}\sim 0$.
\end{introthm}

\subsection{Index of semi-log canonical pairs} 

In order to establish 
the applications to mirror symmetry in \S\ref{sec:ms},
we will need to prove a version
of Theorem~\ref{introthm:main-thm}
for semi-log canonical pairs; see \cite[\S5]{Kol13} for a definition of these pairs. 
The key idea
is to produce a section
on a suitable normal modification of the pair
which is invariant under certain birational transformations.
These sections are called \emph{admissible}; see Definition~\ref{def:admissible}.
We obtain the following Theorem \ref{introthm:admissible}
regarding the existence of 
admissible sections on
log Calabi--Yau pairs of coregularity $0$. To this end, given a pair $(X, B)$, we denote by ${\rm Bir}(X,B)$ the group of crepant birational automorphisms of $(X,B)$; see also~\cite[Definition 1.1]{HX16}.

\begin{introthm}\label{introthm:admissible}
Let $(X,B)$ be a projective 
log Calabi--Yau
pair of coregularity $0$ and Weil index $\lambda$. Set $\lambda'=\lcm(\lambda,2)$.
Then, there is a non-trivial section
$s\in H^0(X,\mathcal{O}_X(\lambda'(K_X+B)))$
such that $g^*s=s$ for every $g\in {\rm Bir}(X,B)$.
\end{introthm}

Theorem~\ref{introthm:admissible} may be regarded
as an effective version of the boundedness
of B-representations; see~\cite[Theorem 1.2]{HX16} or ~\cite[Theorem 3.15]{FG14}.
As a consequence, we show the following corollary.

\begin{introcor}\label{introthm:slc} 
Let $(X,B)$ be a projective
semi-log canonical Calabi--Yau
pair of coregularity $0$ and Weil index $\lambda$.
Then, we have that $\lambda'(K_X+B)\sim 0$,
where $\lambda'=\lcm(\lambda,2)$.
\end{introcor}

Corollary \ref{introthm:slc}
is an immediate consequence of Theorem~\ref{introthm:main-thm}
and Theorem~\ref{introthm:admissible}.
Our techniques do not apply to the case of
generalized semi-log canonical pairs; see Remark~\ref{rem:gslc}.

\subsection{Index of log canonical singularities}
By adjunction, the exceptional divisor of a dlt modification of an isolated log canonical singularity is a projective semi-log canonical log Calabi--Yau pair.
Therefore, Corollary \ref{introthm:slc} provides a bound on the index of log canonical singularities.

\begin{introcor}\label{cor:sing}
Let $(X, B; \, x)$ be the germ of a log canonical pair of coregularity 0 and Weil index $\lambda$.
Suppose that the pair $(X, B)$ is  dlt away from $x$.
Then, we have that $\lambda'(K_X+B)\sim 0$,
where $\lambda'=\lcm(\lambda,2)$.
\end{introcor}

If the pair $(X, B; \, x)$ has only standard coefficients, one can drop the assumption that the log canonical germ $(X,B;\,x)$ is dlt away from $x$, as shown by Fujino in \cite[Theorem 4.18]{Fuj01}.
In particular, Corollary~\ref{cor:sing} and \cite[Theorem 4.18]{Fuj01} specialize to Corollary~\ref{cor:isolatedlogcan}, which was first proved by Ishii in \cite[Theorem 4.5]{Ish00}.
Note that the Hodge theoretic assumption in {\it loc.\ cit.} coincides with our condition of coregularity 0; see \cite[Claim 32.3]{KX16}. 

\begin{introcor}\label{cor:isolatedlogcan}
The index of an isolated log canonical singularity $(X;\, x)$ of coregularity 0 is either 1 or 2.
\end{introcor}

The only possible log canonical surface singularities of coregularity 0 and index 2 are double quotients of cusp singularities; see, e.g., \cite[\S 7.8]{Ishii2018}.
Examples in dimension 3 appear in \cite[Proposition 3.54 and Proposition 3.59]{Kol13} provided that the 2-manifold $F$ in {\it loc.\ cit.} is chosen to be non-orientable. 

\subsection{Calabi--Yau and holomorphic symplectic varieties with group actions} We apply Theorem \ref{introcor-GS} to study degenerations of quotients of Calabi--Yau varieties, and in particular of holomorphic symplectic varieties.

\begin{introthm}\label{thm_group_actions_intro}
Let $X$ be a projective log canonical variety with $K_X \sim 0$ and let $G$ be a finite group acting on $X$ freely in codimension 1.
Assume that the quotient $X/G$ is a fiber of a minimal family of Calabi--Yau pairs of coregularity $0$. 
Then, the order of the character $\rho \colon G \to \mathrm{GL}(H^0(X,\O X.(K_X)))$ is at most $2$.
\end{introthm}
One may drop the assumption that $G$ acts freely in codimension 1 and replace $X$ with a log Calabi--Yau pair with standard coefficients. For expository reasons, we postpone it to Theorem~\ref{thm_group_actions_general}. 
In particular, the result imposes strong constraints on degenerations of holomorphic symplectic varieties with a non-symplectic automorphism.
See \S\ref{sec:quotients} for definitions.

\begin{introcor}\label{thm_hyperkahler}
Let $\pi^* \colon \mathcal{X}^{*} \to C^*$ be projective family of type III of holomorphic symplectic varieties of dimension $2n$ with a finite order automorphism $g \colon \mathcal{X}^{*} \to \mathcal{X}^{*}$ acting purely non-symplectically on the fibers $\mathcal{X}^{*}_{c}$ for any $c \in C^*$. Then the order of $g$ divides $2n$.
\end{introcor}

\begin{remark} \label{rmk_alexeev}
The case of K3 surfaces is proved in \cite[Corollary 3.16]{AEC2021} or \cite{Matsumoto2016} via Hodge theory.
The arguments of \cite[\S3]{AEC2021} extend verbatim to the case of higher dimensional primitive symplectic varieties, see Definition \ref{def:holosympl}.
This means that if the fibers $\mathcal{X}^*_{c}$ in Corollary~\ref{thm_hyperkahler} are primitive symplectic, the order of $g$ is exactly $2$.
\end{remark}

\subsection{Strategy of the proof of Theorem \ref{introthm:main-thm}}
We distinguish two cases: either $B \sups <1. + \bM X.$ is numerically trivial, or not. In the first case, we can actually assume that $B^{=1}=B$ and $\bM .=0$, as explained in \S\ref{num_trivial_subsection}, and so Theorem \ref{introthm:main-thm} reduces to Proposition~\ref{cor:reducedCY}.
In \S\ref{sec:orientability} and \S\ref{sec:dualcomplexcoreg0} we provide a (self-contained) topological proof of Proposition~\ref{cor:reducedCY}. We explain the relation between the index of a log Calabi--Yau pair and the orientability of its dual complex.
In brief, we show that the cohomology group $H^0(X, \mathcal{O}_{X}(K_X+B))$ can be endowed with an integral structure given by the top cohomology of the dual complex $\mathcal{D}(B)$.
The factor $2$ in the index then serves to control the orientation of $\mathcal{D}(B)$. 

Following a suggestion of Koll\'{a}r, we also present a second proof of Proposition~\ref{cor:reducedCY} in \S\ref{sec:residue}.
In this case, the presence of the factor $2$ reflects the fact that Poincar\'{e} residue maps are defined only up to a sign and the fact that the rational canonical form $dx/x$ on $\pp^1$ has residues $1$ and $-1$ at its poles. 
See Remark~\ref{rem:comparison} for a comparison between the two approaches.

If instead $B \sups <1. + \bM X. \nequiv 0$, we adopt an inductive strategy from birational geometry.
First, by the work of Filipazzi and Svaldi, we may assume that $B^{=1}$ fully supports a big divisor; see~\cite{FS20}.
Then, we can run a $(K_X + B^{=1})$-MMP
which terminates with a Mori fiber space $Y\rightarrow Z$.
Let $B_Y$ be the push-forward of $B$ on $Y$. 
Since $(Y,B_Y,\bM.)$ and $(X,B,\bM .)$ are crepant equivalent, they have the same index; see Corollary~\ref{lem:reddlt}. 
Thus, it suffices to control the index
of $(Y,B_Y,\bM.)$.
The divisor $B^{=1}$ fully supports a big divisor, so
some component $S$ of ${B_Y}^{=1}$ is ample over the base.
Since $(Y,B_Y)$ is dlt, such component $S$ is normal.
We can perform adjunction to $S$ and 
obtain a generalized log Calabi--Yau pair structure
$(S,B_S,\bN .)$. 
A careful analysis of the coefficients of $B_S$ and $\bN.$ shows that 
$\lambda'(K_S+B_S+\bN S.)$ is Weil, where $\lambda'=\lcm(\lambda,2)$; see \S\ref{subsec:Weil-index}.
By adjunction, $(S,B_S,\bN.)$ is a generalized log Calabi--Yau pair.
By induction on the dimension, we conclude that
$\lambda'(K_S+B_S+\bN S.) \sim 0$.
Then, using the positivity of $S$ and a relative version of Kawamata--Viehweg vanishing over $Z$, we conclude that $\lambda'(K_X+B+\bM X.)\sim 0$; see \S\ref{sec:fractionalpart}.

\begin{remark}\label{rem:gslc}
In this work, we avoid the use of non-normal generalized pairs. 
In order to control the index of non-normal pairs, it is often necessary to use gluing theory. 
For recent advances in this direction, see for instance \cite{Hu21,JL22}.
\end{remark}

\subsection*{Acknowledgements}
This project was initiated in the \href{https://www.math.ucla.edu/~jmoraga/Learning-Seminar-MMP}{Minimal Model Program Learning Seminar}.
The authors would like to thank J\'{a}nos Koll\'{a}r for suggesting an alternative proof of Proposition~\ref{dlt:reduced}, Valery Alexeev and Giovanni Mongardi for email exchange about  \S\ref{sec:quotients}, and Greg Friedman and Sucharit Sarkar for discussions regarding pseudo-manifolds.

\section{Notation}
We work over a field $\kk$ of characteristic 0.
We refer to \cite{KM98} for the standard terminology in birational geometry.
For the language of generalized pairs and b-divisors, we refer to \cite{FS20}.

Our varieties are connected and quasi-projective
unless otherwise stated.
Given a normal variety $X$ and an open subset $U \subset X$, we say that $U$ is \emph{big} if $\codim_X(X \setminus U) \geq 2$.

In this work, all divisors have coefficients in $\qq$.
Unless otherwise stated, a \emph{divisor} means a Weil $\qq$-divisor.
Given two proper normal varieties $X$ and $X'$ and a $\qq$-Cartier divisor $D$ on $X$ with $D \sim_\qq 0$, the \emph{crepant trace} $D'$ of $D$ is defined as follows:
we take a common resolution $X''$ with morphisms $p \colon X'' \rar X$ and $q \colon X'' \rar X'$.
Then, we set $D' \coloneqq q_*p^* D$.
If $(X,B,\bM .)$ is a generalized log Calabi--Yau pair, its trace $(X',B',\bM.)$ on $X'$ is defined in the same fashion, where we put $D= \K X. + B  + \bM X.$.

\begin{definition}[Dual complex]\label{def:dualcomplex}
Let $(X,B)$ be a dlt pair 
with
$B^{=1}=\sum_{i\in I}B_i$.
The \emph{dual complex} of $(X,B)$, denoted by $\mathcal{D}(B)$, is the regular $\Delta$-complex whose vertices are in correspondence with the irreducible components of $B^{=1}$ and whose $k$-faces correspond to connected components of $B_{i_0} \cap \ldots \cap B_{i_k}$.
\end{definition}

\begin{definition}
\label{rmk:dualcomplex} Let $(X,B,\bM.)$ be a generalized log Calabi--Yau pair (not necessarily dlt). Then $\mathcal{DMR}(X,B,\bM.)$ is the PL-homeomorphism type of the dual complex of a generalized dlt modification $(X',B',\bM.)$ of $(X,B,\bM.)$, i.e.,
\[\mathcal{DMR}(X,B,\bM.) \simeq_{\mathrm{PL}}\mathcal{D}(B').\]
\end{definition} 

\begin{remark}
The notation $\mathcal{DMR}$ stands for “Dual complex of a Minimal divisorial log terminal partial Resolution”. The definition is independent of the choice of the dlt modification by \cite[Proposition 11]{dFKX17}, up to PL-homemorphism. In Berkovich geometry, the cone over $\mathcal{DMR}(X,B)$ is also called (essential) skeleton of $(X,B)$; see, e.g., \cite[\S 3]{MMS2022}.
\end{remark}

\begin{definition}[Coregularity]\label{defn:coreg}
Let $c$ be a natural number. 
A generalized log canonical pair $(X,B,\bM.)$ has \emph{coregularity $c$} if some (any) generalized dlt modification $(X',B',\bM.)$ of $(X,B,\bM.)$ has a generalized log canonical center of dimension $c$, and $c$ is minimal among the dimensions of the generalized log canonical centers of $(X',B',\bM.)$.
Equivalently, $(X,B,\bM.)$ has coregularity $c$ if the dual complex $\mathcal{DMR}(X,B,\bM.)$ has dimension $\dim X -c -1$.
\end{definition}

We refer to \cite{FMP22, Mor22} for more details about the notion of coregularity.

\section{Preliminaries}

In this section, we collect some preliminaries about Cartier indices, Weil indices of generalized pairs, and numerically trivial moduli parts.
We prove some related technical statements that will be used in the proof of the main theorem.

\subsection{Cartier index of a numerically trivial divisor}

In this subsection, we prove some statements regarding the Cartier index of a numerically trivial divisor.

\begin{lemma}\label{lemma_Weil_torsion}
Let $D \sim_{\qq}0$ be a numerically trivial integral $\qq$-Cartier divisor on a normal proper variety $X$. If $H^0(X, \mathcal{O}_{X}(D))\neq 0$, then $D$ is Cartier and $D \sim 0$.
\end{lemma}

\begin{proof}
Let $s \in \Gamma(X,\O X.(D))$ be a non-zero section.
Let $U$ be any smooth big open subset of $X$.
In particular, $D$ is Cartier along $U$.
Since $D \sim_{\qq} 0$ and $X$ is integral, $s$ is a nowhere vanishing section on $U$, inducing a trivialization of $\O X.(D)$ along $U$.
Since $\O X.(D)$ is a reflexive sheaf and $U$ is a big open subset, $s$ is a nowhere vanishing section trivializing $\O X.(D)$.
Thus, $D$ is a linearly trivial Cartier divisor.
\end{proof}

The following lemma shows that the Cartier index of a $\qq$-trivial $\qq$-Cartier divisor is independent of the birational model.
In particular, in order to control the index of a log Calabi--Yau pair, we are allowed to take dlt modifications and run minimal model programs.

\begin{lemma} \label{index_crepant_trace}
Let $X$ be a proper
normal variety, and let $D$ be a $\qq$-Cartier divisor with $D \sim _\qq 0$.
Also, let $X'$ be a proper normal variety that is birational to $X$, and let $D'$ be the crepant trace of $D$ on $X'$.
Then, for any integer $m$, $mD \sim 0$ if and only if $mD' \sim 0$.
\end{lemma}
\begin{proof}
Let $X''$ be a common resolution of $X$ and $X'$, and let $p \colon X'' \rar X$ and $q \colon X'' \rar X'$ be the corresponding morphisms.
Let $D''$ be the crepant trace of $D$ on $X''$.
By the symmetry of the problem, it suffices to show the following:
if $m D \sim 0$ holds, then so does $m D' \sim 0$.
Fix such $m$.
By the projection formula, the linear equivalence is preserved by pull-back, and we have $mD'' \sim 0$.
Then, the linear equivalence is preserved by the push-forward $q_*$, as $q$ is a birational morphism of normal varieties.
Indeed, we may find a big open subset $U' \subset X'$ over which $q$ is an isomorphism.
Then, as $m D'' \sim 0$, we have $m D' \sim 0$ along $U'$.
Finally, by the $S_2$ property, $m D' \sim 0$ holds true on the whole $X'$.
\end{proof}

We have the following immediate corollary.

\begin{corollary}\label{lem:reddlt}
Let $(X',B',\bM.)$ be a proper generalized log Calabi--Yau pair that is crepant birational to the proper generalized log Calabi--Yau pair $(X,B,\bM.)$.
Then, for any integer $m$, $m(K_{X'}+B'+\bM X'.)\sim 0$ if and only if $m(K_X+B+\bM X.)\sim 0$.
\end{corollary}

\subsection{Weil index of a generalized pair}
\label{subsec:Weil-index} 

In this subsection, we study the Weil index of coregularity 0 log Calabi--Yau pairs under adjunction.

\begin{definition}[Weil index of a generalized pair]\label{def:Weilindex}
Let $(X,B,\bM.)$ be a generalized log canonical pair. Let $\Lambda \subset \qq$ be the smallest set of positive rational numbers
such that
\begin{itemize}
    \item the coefficients of $B$ are contained in $\Lambda$; and
    \item we can write $\bM X'.=\sum_j \lambda_j M_j$, where $X'\rightarrow X$ is a model where 
    $\bM.$ descends, 
    $\lambda_j\in \Lambda$,
    and each $M_j$ is nef Cartier.
\end{itemize}
Then, the \emph{Weil index of the generalized pair $(X,B,\bM.)$} is the smallest positive integer for which $\lambda \Lambda \subset \nn$.
\end{definition}

\begin{example}
The Weil index of the log canonical pair $(X, B = \sum_{i} \frac{p_i}{q_i}B_i)$, whose coefficients $\frac{p_i}{q_i}$ are irreducible fractions is the least common multiple of the denominators $q_i$.
\end{example}

We show that the Weil index
of a log pair of coregularity 0 is stable under adjunction.
To this end, we will introduce some notation
and recall a statement from~\cite{FMP22}.

\begin{notation}\label{not:1}
Let $\Lambda$ be a finite subset of $[0,1]$, and $r$ be a non-negative rational number. We define the set
\[
D_\Lambda(r) \coloneqq \left\{ 
1-\frac{1}{m} + \frac{m_0r+\sum_{i=1}^k m_i\lambda_i}{m} \mid \lambda_i \in \Lambda \cup \{0, 1\}, \text{ and }
 m,m_i, k\in\zz_{> 0}
\right\} \cap [0,1].
\]
If $r=0$, then we set $D_\Lambda \coloneqq D_\Lambda(0)$.
If $\Lambda = \zz\left[\frac{1}{\lambda}\right]\cap [0,1]$ for some positive integer $\lambda$, then we denote the set $D_\Lambda(r)$ simply by $D_\lambda(r)$.
\end{notation}

The following lemma follows from~\cite[Lemma 3.2 and Remark 3.3]{FMP22}.

\begin{lemma}\label{lem:producing-lcy-p1}
Let $(X,B,\bM.)$ be a projective generalized log Calabi--Yau pair of coregularity $0$. Let $\Lambda$ be a finite subset of $[0,1]$, and $r$ be a non-negative rational number.
Assume that the coefficients of $B$ belong to $D_\Lambda$, and
$r$ is a coefficient of $B$.

Then, there exists a generalized log Calabi--Yau pair $(\pp^1,B_{\pp^1},\bN.)$
satisfying the following conditions:
\begin{itemize}
    \item the coefficients of $B_{\pp^1}$ belong to $D_\Lambda$;
    \item there is $q\in \pp^1$ such that
    ${\rm coeff}_q(B_{\pp^1})=1$; and 
    \item there is $p\in \pp^1$, with $p\neq q$, for which 
    ${\rm coeff}_p(B_{\pp^1}) \in D_\Lambda(r)$.
\end{itemize}
In particular, the following facts hold: 
\begin{enumerate}
    \item[(i)] if $\bM. =0$, then we can choose $\bN.=0$ as well; and
     \item[(ii)] if there exists a positive integer $\lambda$ such that $\Lambda \subseteq \zz\left[\frac{1}{\lambda}\right]\cap [0,1]$ and $\lambda \bM.$ is b-Cartier, then $\lambda \bN.$ is b-Cartier as well.
\end{enumerate}
\end{lemma}

Now, we turn to prove the statement about the Weil index under adjunction.

\begin{lemma}\label{lem:weil-index-under-adjunction}
Let $(X,B,\bM.)$ be a generalized
log canonical pair of Weil index $\lambda$.
Let $\lambda' \coloneqq \lcm(\lambda,2)$.
Let $S$ be a prime component of  $B \sups =1.$.
Assume that $S$ is projective.
Let $S^\nu \rightarrow S$ be its normalization.
Let $(S^\nu,B_{S^\nu},\bN .)$ be the generalized pair obtained by generalized adjunction.
Assume that $(S^\nu,B_{S^\nu},\bN.)$
is generalized log Calabi--Yau of coregularity $0$.
Then, the Weil index of the generalized pair $(S^\nu,B_{S^\nu},\bN .)$ divides $\lambda'$.

In particular, if $(X,B,\bM.)$ is a generalized log Calabi--Yau pair of Weil index $\lambda$ and coregularity $0$, then
the Weil index of $(S^\nu,B_{S^\nu},\bN.)$ divides $\lambda'$.
\end{lemma}

\begin{proof}
Up to passing to a dlt modification, we can assume that $(X,B,\bM.)$ is generalized dlt
and that $S=S^\nu$.
Indeed, the Weil index of a generalized log canonical pair coincides with the Weil index of any crepant dlt modification, and any dlt modification of $(X,B,\bM.)$ restricts to a dlt modification of $(S,B_S,\bN .)$. Notice that we may assume that the dlt modification is projective over the original model, thus the projectivity of $S$ is preserved.
Also, the statement is clear for what concerns $\bN.$, since it is the restriction of the b-divisor $\bM.$.
Thus, in the rest of the proof, we will focus on the boundary divisor $B_S$.
By adjunction, the coefficients of $B_S$ belong to $D_\lambda$; see ~\cite[proof of Lemma 3.3]{Bir19} where $p$ and $\mathfrak{R}$ in {\it loc. cit.} stand for $\lambda$ and $\zz[1/\lambda] \cap [0,1]$, respectively.
Now, let $P$ be a prime divisor on $S$.
We can write 
\[
{\rm coeff}_P(B_S)=\frac{p_0}{\lambda_0},
\]
where the fraction is irreducible.
We show that $\lambda_0$ divides $\lambda'$. 
If $X$ is smooth at $P$, then the statement is clear.
Thus, we have that 
${\rm coeff}_P(\mathrm{Diff}_S(0))=1-\frac{1}{m}$ for some integer $m\geq 2$, see~\cite[\S4.1]{Kol13}.
Then, we have that 
\[
{\rm coeff}_P(B_S) \geq 
{\rm coeff}_P(\mathrm{Diff}_S(0)) = 1-\frac{1}{m} \geq \frac{1}{2}.
\] 
If $\frac{p_0}{\lambda_0}=\frac{1}{2}$, then
$\lambda_0=2$ divides $\lambda'$.
From now on, we assume that $\frac{p_0}{\lambda_0}>\frac{1}{2}$.
We apply Lemma~\ref{lem:producing-lcy-p1} with $r=\frac{p_0}{\lambda_0}$.
Then, we can find a generalized pair structure
$(\pp^1, B_{\pp^1}+M_{\pp^1})$
and two distinguished points $p,q\in \pp^1$ satisfying the following conditions:
\begin{itemize}
    \item the Weil index of $M_{\pp^1}$ is $\lambda$;
    \item the divisor $B_{\pp^1}$ has coefficient 1 at $q$;
    \item we have that 
    \[
    {\rm coeff}_p(B_{\pp^1}) = 1-\frac{1}{m_p} + \frac{ \frac{ p_0}{\lambda_0}+ \sum_{i=1}^k \frac{m_ip_i}{\lambda}}{m_p};
    \]
    where $m_i,m_p, k \in \zz_{>0}$, and $p_i$ are non-negative integers with $0 \leq p_i \leq \lambda$; and 
    \item for every $s\in \pp^1 \setminus \{p,q\}$, we have that
    \[
    {\rm coeff}_s\left(B_{\pp^1}\right)=
    1-\frac{1}{m_s}+\frac{\sum_{i=1}^k\frac{m_ip_i}{\lambda}}{m_s},
    \] 
    where $m_i,m_p, k \in \zz_{>0}$, and $p_i$ are non-negative integers with $0 \leq p_i \leq \lambda$.
\end{itemize}
If $m_{s_0}\geq 2$ for some
$s_0\in \pp^1 \setminus \{p,q\}$, then we would have
\[
{\rm deg}(B_{\pp^1}) \geq 
{\rm coeff}_q(B_{\pp^1}) + 
{\rm coeff}_p(B_{\pp^1}) +
{\rm coeff}_{s_0}(\pp^1) >
1 + \frac{1}{2} + \frac{1}{2} = 2,
\] 
leading to a contradiction. 
Thus, we have that 
$m_s=1$ for every $s\in \pp^1 \setminus \{p,q\}$.
Taking the degree of $K_{\pp^1}+B_{\pp^1}+M_{\pp^1}$, we obtain a relation of the form
\[
    1-\frac{1}{m}+\frac{\frac{p_0}{\lambda_0} + \sum_{i=1}^k\frac{m_ip_i}{\lambda}}{m} + \sum_{j=1}^\ell \frac{n_jp'_j}{\lambda}=1, 
\]
where each $p_i$ and each $p'_j$
are positive integers. 
Hence, we may write
\[
p_0\lambda_0^{-1} = \left(\lambda-\sum_{i=1}^k m_ip_i -\sum_{j=1}^\ell mn_jp'_j\right)\lambda^{-1}.
\]
We conclude that $\lambda_0$ divides $\lambda$, and so $\lambda'$ does too. This finishes the proof.
\end{proof}

\subsection{Numerically trivial moduli part}\label{num_trivial_subsection}
In this subsection, we show two lemmas regarding generalized pairs with numerically trivial moduli part.

\begin{lemma}\label{cor_0_pic0}
Let $(X,B,\bM.)$ be a projective generalized log Calabi--Yau pair.
Assume one of the following two conditions:
\begin{itemize}
    \item[(i)] $(X,B,\bM.)$ has coregularity 0; or
    \item[(ii)] $X$ is rationally connected and of klt-type.
\end{itemize}
Let $c$ be the b-Cartier index of $\bM.$.
If $\bM X.$ is $\qq$-Cartier with $\bM X.\equiv 0$, then $c\bM. \sim 0$ as b-divisor.
\end{lemma}

\begin{proof}
Let $\pi \colon X' \rar X$ be a projective resolution of singularities of $X$ where $\bM.$ descends.
By the negativity lemma \cite[Lemma 3.39]{KM98}, we have $\bM X'. =\pi^*\bM X. \equiv 0$.
Let $c$ be the Weil index of $\bM X'.$.
Since $X'$ is smooth, then $c\bM X'.$ is Cartier.
By \cite[Theorem 4.2]{FS20} and \cite{HMc07} (note that just \cite{HMc07} suffices for case (ii)), $X'$ is rationally connected, so $H^1(X', \mathcal{O}_{X'})=0$ and $\Pic(X') = \mathrm{NS}(X')$. 
This implies that $c\bM X'.$ is a torsion Cartier divisor, i.e., $c\bM X'. \sim_{\qq} 0$.
Since a rationally connected variety has no non-trivial \'{e}tale covers (see, e.g., \cite[Corollary 4.18.(b)]{Debarre2001}), the \'{e}tale cover associated with $c\bM X'. \sim_{\qq} 0$ must be trivial, which means that $c\bM X'. \sim 0$.
\end{proof}

We have the following immediate corollary.

\begin{corollary} \label{M:torsion}
Let $(X,B,\bM.)$ be a projective generalized log Calabi--Yau pair of coregularity 0.
Assume that $\bM X.$ is $\qq$-Cartier with $\bM X. \equiv 0$.
Let $\lambda$ be a positive integer such that $\lambda \bM.$ is b-Cartier.
Then, Theorem \ref{introthm:main-thm} holds for $(X,B,\bM.)$ and $\lambda$ if and only if so does for the pair $(X,B)$ and $\lambda$.
\end{corollary}

\begin{proof}
By Lemma \ref{cor_0_pic0}, we have $\lambda \bM. \sim 0$ as b-divisor.
Then, the claim follows, as the statement of Theorem \ref{introthm:main-thm} concerns linear equivalence up to multiplication by $\lambda'=\lcm(\lambda,2)$.
\end{proof}

\subsection{Coregularity and finite quotients}\label{quotients_subsection}
In this subsection, we show that the coregularity is preserved by finite maps.

\begin{proposition}\label{prop_quotient_coreg}
Let $(X,\Delta_X)$ and $(Y,\Delta_Y)$ be two log canonical pairs with a generically finite surjective morphism $f \colon (X,\Delta_X) \rar (Y,\Delta_Y)$.
Assume that $K_X + \Delta_X=f^*(K_Y+\Delta_Y)$.
Then, we have $\mathrm{coreg}(X,\Delta_X)=\mathrm{coreg}(Y,\Delta_Y)$.
\end{proposition}

\begin{proof}
Throughout the proof, given a log canonical pair $(W,\Gamma)$, we denote by $\mathrm{dmLCC}(W,\Gamma)$ the minimum among the dimensions of the log canonical centers of $(W,\Gamma)$. Note that \begin{equation}\label{eq:dmlcoreg}\mathrm{dmLCC}(W,\Gamma) \leq \mathrm{coreg}(W,\Gamma),\end{equation} and the equality holds if $(W,\Gamma)$ is dlt. In particular, if $(W,\Gamma)$ and $(W',\Gamma')$ are crepant birational dlt pairs, then 
\begin{equation}\label{eq:dmlII}
\mathrm{coreg}(W,\Gamma)=\mathrm{dmLCC}(W,\Gamma)=\mathrm{dmLCC}(W',\Gamma')=\mathrm{coreg}(W',\Gamma')\end{equation} by \cite[Proposition 11]{dFKX17}.

Since the coregularity is preserved by crepant birational morphisms, we may replace $(X,\Delta_X)$ with the pair induced on the Stein factorization of $X \rar Y$.
In particular, we may assume that $f$ is finite.

Let $(Y',\Delta_{Y'})$ be a dlt modification of $(Y,\Delta_Y)$, and let $X'$ denote the normalization of the main component of $X \times_Y Y'$, and let $\pi \colon X' \rar X$ be the induced morphism.
Define $K_{X'} + \Delta_{X'}=\pi^*(K_X+\Delta_X)$.
Then, by \cite[Proposition 5.20]{KM98}, $(X',\Delta_{X'})$ is a pair.
Since the coregularity is preserved by crepant birational morphisms, we may replace $(X,\Delta_X)$ and $(Y,\Delta_Y)$ with $(X',\Delta_{X'})$ and $(Y',\Delta_{Y'})$, respectively.
So, in the following, we may assume that $(Y,\Delta_Y)$ is dlt.

Now, let $(\tilde X, \Delta_{\tilde{X}})$ denote a dlt modification of $(X,\Delta_X)$, and let $(\tilde{Y},\Delta_{\tilde Y})$ denote the pair induced on the Stein factorization of $\tilde X \rar Y$.

Since $(\tilde X,\Delta_{\tilde X}) \rar (X ,\Delta_X)$ and $(\tilde Y,\Delta_{\tilde Y}) \rar (Y ,\Delta_Y)$ are crepant birational and by \cite[Proposition 5.20]{KM98}, we have the inequalities
\[
\mathrm{dmLCC}(Y,\Delta_Y) = \mathrm{dmLCC}(X,\Delta_X) \leq \mathrm{dmLCC}(\tilde X, \Delta_{\tilde X}) = \mathrm{dmLCC}(\tilde Y, \Delta_{\tilde Y}).
\]
By \eqref{eq:dmlcoreg} and \eqref{eq:dmlII}, since $(Y, \Delta)$ is dlt and crepant to $(\tilde Y, \Delta_{\tilde Y})$, we have \[\mathrm{dmLCC}(Y,\Delta_Y)\geq \mathrm{dmLCC}(\tilde Y, \Delta_{\tilde Y}).\]
Putting together the two chains of inequalities, we conclude that 
\[\mathrm{coreg}(Y, \Delta_Y)=\mathrm{dmLCC}(Y,\Delta_Y) = \mathrm{dmLCC}(X,\Delta_X) = \mathrm{dmLCC}(\tilde X, \Delta_{\tilde X}) = \mathrm{coreg}(X, \Delta_X).\]
\end{proof}

If $f\colon (X,\Delta_X) \rar (Y,\Delta_Y)$ is a finite Galois morphism, the following stronger version of Proposition \ref{prop_quotient_coreg} holds. See Definition \ref{rmk:dualcomplex} for the definition of $\mathcal{DMR}(X,\Delta_X)$.

\begin{proposition}\label{prop:dualcomplexquotient}
Let $f \colon (X,\Delta_X) \rar (Y,\Delta_Y)$ be a finite Galois morphism of log canonical pairs $(X,\Delta_X)$ and $(Y,\Delta_Y)$ with Galois group $G$.
Assume that $K_X + \Delta_X=f^*(K_Y+\Delta_Y)$. Then, the group $G$ acts by PL-homeomorphisms on the dual complex $\mathcal{D}(X,\Delta_X)$, and there exists a PL-homeomorphism 
\[\mathcal{DMR}(Y,\Delta_Y) \simeq_{\mathrm{PL}}\mathcal{DMR}(X,\Delta_X)/G.\]
In particular, we have $\mathrm{coreg}(X,\Delta_X)=\mathrm{coreg}(Y,\Delta_Y)$.
\end{proposition}
\begin{proof}
The same argument of \cite[Theorem F]{MMS2022} continues to work if we replace the assumption in {\it loc.\ cit.}, namely that $f$ is not ramified in codimension $1$, with the current crepant assumption $K_X + \Delta_X=f^*(K_Y+\Delta_Y)$. 
\end{proof}

\section{Fractional and moduli divisors}\label{sec:fractionalpart}

In this section, we prove an inductive statement for the main theorem when the fractional or moduli parts are non-trivial.

\begin{proposition} \label{fractional:part}
Let $(X,B,\bM.)$ be a projective generalized log Calabi--Yau pair of coregularity 0 and Weil index $\lambda$.
Suppose that $B^{< 1}+\bM X. \nequiv 0$ and that Theorem \ref{introthm:main-thm} holds in dimension less than $\dim X$.
Then $\lambda' (K_X + B +\bM X.) \sim 0$, where $\lambda'=\lcm(\lambda,2)$.
\end{proposition}

\begin{proof}
We proceed by induction on the dimension.
The base case is $\dim X=1$;
by the assumption $B^{< 1}+\bM. \nequiv 0$, it follows that $X = \pr 1.$.
Then, the claim follows immediately.
Thus, in the rest of the proof, we will address the inductive step, and we will assume that Theorem \ref{introthm:main-thm} is known in lower dimension.

By Corollary \ref{lem:reddlt}, we may replace $(X,B,\bM.)$ in its crepant birational class, as long as we control the Weil index of the new pair.
Thus, by taking a $\qq$-factorial generalized dlt modification, then applying \cite[Theorem 4.2]{FS20}, and finally by taking again a $\qq$-factorial generalized dlt modification, we may assume that $X$ is $\qq$-factorial, $(X,B,\bM.)$ is generalized dlt, and $B\sups =1.$ fully supports a big divisor.
Notice that, by the properties of the birational transformations in \cite[Theorem 4.2]{FS20} and dlt modifications, all the new divisors extracted in the process appear in $B$ with coefficient 1.
Thus, the Weil index $\lambda$ remains unchanged by this reduction.
Similarly, the property $B \sups <1. + \bM X. \nequiv 0$ is preserved.

Now, we run a $(K_X + B^{=1})$-MMP with scaling, which terminates with a Mori fiber space $g \colon Y \to Z$, as $B \sups <1. + \bM X. \nequiv 0$.
Let $(Y, B_Y,\bM.)$ be the induced generalized log Calabi--Yau pair, which is not necessarily generalized dlt.
We observe that the pair $(Y, B^{=1}_{Y})$ is dlt, and $B^{=1}_{Y} \neq 0$ dominates $Z$ since $B^{=1}_{Y}$ supports a big divisor.
Furthermore, the Weil index $\tilde \lambda$ of $(Y,B_Y,\bM.)$ divides the Weil index $\lambda$ of $(X,B,\bM.)$, since we may have contracted some components of $B \sups <1.$.

Let $S$ be an irreducible component of $B^{=1}_{Y}$ dominating $Z$.
Notice that $S$ is normal, since $(Y, B^{=1}_{Y})$ is dlt.
The generalized pair $(S, B_{S}, \bN.)$ obtained by generalized adjunction
\begin{equation}\label{eq:adjunction}
    K_{S} + B_S + \bN S. \sim_{\qq} (K_Y+ B_Y+\bM Y.)|_{S}
\end{equation}
is log Calabi--Yau, and its Weil index divides $\lambda'=\lcm(\lambda,2)$ by Lemma \ref{lem:weil-index-under-adjunction}.
By induction on the dimension, we obtain 
\[
\lambda' (K_{S} + B_S + \bN S.) \sim 0.
\]

The coefficients of $B_S$ and $\bN S.$ control the coefficients of $\mathrm{Diff}_S(0)$, as we explain in what follows.
By \cite[\S 3.35]{Kol13}, at the codimension 2 points of $X$ contained in $S$, $X$ has cyclic singularities.
Then, given a prime divisor $P$ in $S$, an \'{e}tale local neighbourhood of a general point $p \in P$ is isomorphic to
\[
(p \in (X, B, \bM.))\simeq (0 \in (\Af^2=(x,y), (x=0) + c(y=0)))/(\zz/m\zz)\times \Af^{\dim X -2},
\]
where $S=(x=0)$ and $S'=(y=0)$.
Since the class group of $Z \coloneqq \Af^2/(\zz/m\zz)$ is generated by $S'$, there exists an integer $\mu$ such that locally
\begin{equation}\label{eq:linear-equiv-toric}
\lambda'(K_X+B+\bM X.) \sim \mu S'.
\end{equation}
By adjunction, $S'|_S \sim \frac{1}{m}\{0\}$.
Since the denominators of the coefficients of $B_S$ and $\bN S.$ divide $\lambda'$, $\lambda'(K_{S} + B_S + \bN S.)$ is a Weil divisor on the smooth germ $S$ of $p$, and it is $\qq$-linearly trivial by \eqref{eq:adjunction}.
Therefore, $\lambda'(K_{S} + B_S + \bN S.)$ is actually a trivial Cartier divisor at $p$, and $\lambda'(K_X+B+\bM X.)|_S\sim 0$. 
Since we write
\[
0 \sim \lambda'(K_X+B+M)|_S \sim 
\mu S'|_S \sim \frac{\mu}{m}\{0\},
\] 
we conclude that $m$ divides $\mu$.
In particular, we have that $\mu S'$ is a Cartier divisor, as $m$ is the order of the class group of $X$ at $\{0\}$.
By the linear equivalence~\eqref{eq:linear-equiv-toric}, we conclude that the divisor
$\lambda'(K_X+B+M)$ is Cartier in codimension 2. 

Then, by \cite[\S 2.41 and Lemma 2.42]{Bir19}, we have the following short exact sequence
\begin{equation}\label{ses_fiber_space2}
0 \rar \O Y.(\lambda'(K_{Y}+B_Y + \bM Y.)-S) \rar \O Y.(\lambda'(K_{Y}+B_Y+ \bM Y.)) \rar \O S.(\lambda'(K_{S}+B_{S}+ \bN S.)) \rar 0.
\end{equation}
Since $\lambda'(K_{Y}+B_Y + \bM Y.)-S\sim_{\qq,g} -S$, the divisor $-S$ is $g$-ample, and $\dim Z < \dim Y$,
we have
\[
g_*\O Y.(\lambda'(K_{Y}+B + \bM Y.)-S)=0.
\]
Similarly, we write 
\[
\lambda'(K_{Y}+B_Y + \bM Y.)-S\sim_{\qq,g} -S \sim_{\qq,g} K_Y +(B_Y-S+\bM Y.).
\]
Note that $Y$ is klt and $B_Y-S+\bM Y. = B^{<1}_Y+\bM Y. + (B^{=1}_Y - S)$ is $g$-ample, since $g$ is a Mori fiber space and the divisor $B_Y^{<1}+\bM Y.$ is $g$-ample.
Thus, by the relative version of Kawamata--Viehweg vanishing, we have
\[
R^1g_*\O Y.(\lambda'(K_{Y}+B_Y + \bM Y.)-S)=0.
\]
Therefore, by pushing forward \eqref{ses_fiber_space2} via $g$, we obtain 
\[
g_* \O Y.(\lambda'(K_{Y}+B_Y+ \bM Y.)) \simeq g_* \O S.(\lambda'(K_{S}+B_{S}+ \bN S.)).
\]
Now, taking global sections, we have
\begin{equation} \label{eq:sections}
H^0(Y, \O Y.(\lambda'(K_{Y}+B+ \bM Y.))) = H^0(S, \O S.(\lambda'(K_{S}+B_{S}+ \bN S.)))=H^0(S, \O S.) \neq 0.
\end{equation}
By Lemma \ref{lemma_Weil_torsion}, \eqref{eq:sections} implies that $\lambda'(K_{Y}+B_Y+ \bM Y.) \sim 0$.
Therefore, by Corollary \ref{lem:reddlt}, we conclude that $\lambda' (K_X + B_X +\bM X.) \sim 0$.
\end{proof}

\section{Index of log Calabi--Yau pairs and orientability of dual complexes}\label{section reduced} 
In this section, we prove the main theorem of this article.

\subsection{Orientability of dual complexes}\label{sec:orientability}
In this subsection, we recall the concept of pseudo-manifold and dual complex, and we prove some facts about their orientability for log Calabi--Yau pairs. 

\begin{definition}[Pseudo-manifold]\label{defn:pseudo}
A topological space $T$ is called a $n$-dimensional \emph{pseudo-manifold with
boundary} if it admits a triangulation $\mathcal{T}$ satisfying the following conditions:
\begin{enumerate}
    \item (pure dimension) $T = |\mathcal{T}|$ is the union of all $n$-simplices;
    \item (non-branching) every $(n - 1)$-simplex is a face of precisely one or two $n$-simplices; and
    \item (strong connectedness) for every pair of $n$-simplices $\sigma$ and $\sigma'$ in $T$, there is a sequence of $n$-simplices
    \[\sigma = \sigma_0, \sigma_1, \ldots, \sigma_l = \sigma'\]
    such that the intersection $\sigma_i \cap \sigma_{i+1}$ is a $(n-1)$-simplex for all $i$.
\end{enumerate}
The boundary of $T$, denoted by $\partial T$, is the union of all the $(n-1)$-simplices that are faces of only one $n$-simplex.
We say that $T$ is a \emph{closed pseudo-manifold} if $\partial T = \emptyset$.
\end{definition}
\begin{lemma}\label{lem:orientability}
Let $T$ be a $n$-dimensional pseudo-manifold with
boundary, then $H^n(T, \zz) = \zz$ or $0$. If $\partial T \neq \emptyset$, then $H^n(T, \zz)=0$.
\end{lemma}
\begin{proof}
See \cite[\S3.3, p.238]{H02}.
\end{proof}
\begin{definition}[Orientable pseudo-manifold]\label{defn:orient}
A $n$-dimensional closed pseudo-manifold $T$ is \emph{orientable} if $H^n(T, \zz) = \zz$.
\end{definition}

\begin{lemma}
The dual complex $\mathcal{D}(B)$ of a dlt log Calabi--Yau pair $(X, B)$ is a pseudo-manifold with boundary.
\end{lemma}
\begin{proof}
See \cite[Theorem 4.1.4]{NX16}.\footnote{The authors of \cite{NX16} work in the context of degenerations of Calabi--Yau varieties, but their proof works verbatim for log Calabi--Yau pairs.}
\end{proof}

\begin{proposition}[Orientability of dual complexes]\label{lem:pseudomflddualcomplex}
Let $(X, B)$ be a projective dlt log Calabi--Yau pair of coregularity 0 with $B=B \sups =1.$.
Then, the dual complex $\mathcal{D}(B)$ is an orientable closed pseudo-manifold if and only if $K_{X}+ B \sim 0$.
\end{proposition}
\begin{proof}
Set $n+1=\dim X$.
The short exact sequence
\[0 \to \O X.(-B) \to \O X. \to \O B. \to 0\]
induces the long exact sequence in cohomology
\[
H^{n}(X, \O X.) \to H^{n}(B, \O B.) \to H^{n+1}(X, \O X.(-B)) \simeq H^{0}(X, \O X.(K_{X} + B)) \to H^{n+1}(X, \O X.).
\]
Since $(X, B)$ has coregularity 0, $X$ is rationally connected by \cite[\S18]{KX16}, so we have
\[H^{n}(X, \O X.)=H^{n+1}(X, \O X.)=0.\]
Therefore, $K_{X} + B \sim 0$ is equivalent to $H^{n}(B, \O B.) \neq 0$; see also Lemma \ref{lemma_Weil_torsion}. Observe that $H^{n}(\mathcal{D}(B), \kk) \simeq H^{n}(B, \O B.)$ by \cite[Lemma 25]{KX16} and its generalization to the dlt setting in \cite[Proposition A.7]{KLSV2018}. We conclude that the orientability of $\mathcal{D}(B)$ (i.e., $H^{n}(\mathcal{D}(B), \kk) \neq 0$) implies $H^{n}(B, \O B.) \neq 0$, or equivalently $K_{X} + B \sim 0$. The converse follows from \cite[Claim 32.3]{KX16}.
\end{proof}

\subsection{Dual complexes of log Calabi--Yau pairs in coregularity zero}\label{sec:dualcomplexcoreg0}

In this subsection, we prove the statements in \S\ref{sec:MainThm} and \S\ref{sec:orientability:intro} including the main theorem of the paper.

\begin{proof}[Proof of Proposition~\ref{cor:reducedCY}]
Let $m$ be the smallest integer such that $m(K_X+B) \sim 0$. 
Correspondingly, there exists a degree $m$ quasi-\'{e}tale cyclic cover 
\begin{equation}\label{eq:indexone}
    q \colon (\widetilde{X}, \widetilde{B}) \to (X, B)
\end{equation} with $K_{\widetilde{X}} + \widetilde{B} \sim 0$, alias the index one cover of the pair $(X,B)$.
The cover induces a quotient map of regular $\Delta$-complexes  \[\mathcal{D}(\widetilde{B}) \to \mathcal{D}(\widetilde{B})/(\zz/m\zz) \simeq \mathcal{D}(B);\]
see \cite[\S17]{KX16}. Set $\dim X = n+1$ and suppose $m \neq 1$.
By Proposition \ref{lem:pseudomflddualcomplex}, we have \[H^{n}(\mathcal{D}(\widetilde{B}), \zz) = \zz \text{, but }H^{n}(\mathcal{D}(B), \zz)=0.
\]
Note that the Galois group $\zz/m\zz$ of $q$ acts on $ H^{n}(\mathcal{D}(\widetilde{B}), \zz) = \zz$ and exchanges the generators $1$ and $-1$.
Otherwise, we would have
\[\qq = H^{n}(\mathcal{D}(\widetilde{B}), \qq)^{\zz/m\zz} = H^{n}(\mathcal{D}(B), \zz)\otimes \qq =0,
\]
where the second equality follows from \cite[Theorem III.2.4]{Bredon1972}.
Therefore, there exists a group homomorphism $\zz/m\zz \to\zz/2\zz \simeq \{1, -1\}$ that is non-trivial.
In particular, $m$ is even.

Now, let $(Y, B_Y)$ be the dlt pair corresponding to the quotient $\zz/m\zz \to\zz/2\zz$ as in \cite[Theorem 2.(5)]{KX16}. We obtain the sequence of quotient maps
\begin{align*}
    (\widetilde{X}, \widetilde{B}) & \xrightarrow{\frac{m}{2}:1} (Y, B_Y) \xrightarrow{2:1} (X, B)\\
    \mathcal{D}(\widetilde{B}) & \xrightarrow{\frac{m}{2}:1}     \mathcal{D}(B_Y) \xrightarrow{2:1}  \mathcal{D}(B).
\end{align*}
Taking cohomology and using \cite[Theorem III.2.4]{Bredon1972}, we obtain
\[H^{n}(\mathcal{D}(\widetilde{B}), \qq) = \qq \longleftarrow H^n(\mathcal{D}(B_Y), \qq) = H^{n}(\mathcal{D}(\widetilde{B}), \qq)^{\zz/\frac{m}{2}\zz} = \qq \longleftarrow H^{n}(\mathcal{D}(B), \qq) 
= 0, \]
i.e., $K_{Y}+B_Y \sim 0$ by Proposition \ref{lem:pseudomflddualcomplex} and so $m=2$.
The norm of the section trivializing $K_{Y}+B_Y$ descends to a section of $2(K_{X}+B)$.
\end{proof}

\begin{proof}[{Proof of Corollary \ref{cor_orientability}}]
This is immediate from Proposition \ref{lem:pseudomflddualcomplex} and the proof of Proposition \ref{dlt:reduced}.
\end{proof}

\begin{proof}[Proof of Theorem \ref{introthm:main-thm}]
We proceed by induction on the dimension of $X$.
The base case is $\dim X =1$, in which case $X = \pr 1.$ and the statement is clear.
Thus, in the rest of the proof, we will address the inductive step, and we will assume that Theorem \ref{introthm:main-thm} holds in lower dimension.

By Corollary \ref{lem:reddlt}, we may assume that $(X,B,\bM.)$ is generalized dlt.
If $B \sups <1. + \bM X. \nequiv 0$, then the statement follows from Proposition \ref{fractional:part}.
Thus, we may assume that $B \sups <1. + \bM X. \equiv 0$.
Then, by Corollary \ref{M:torsion}, we may assume that $\bM . = 0$ as b-divisor.
But then, we have that $(X,B,\bM.)=(X,B)$ is a dlt pair with $B=B \sups =1.$, and we conclude by Proposition \ref{dlt:reduced}.
\end{proof}

\begin{proof}[Proof of Corollary~\ref{cor:frac1/2}]
Let $\Lambda$ be the set of coefficients of $B$.
By assumption, we have that $\Lambda \subset \left[ \frac{1}{2},1\right]$.
We first show that $2B$ is integral. Then, by Theorem~\ref{introthm:main-thm}, we obtain that $2(K_X+B)\sim 0$.

Since $\Lambda \subset \left[\frac{1}{2},1\right]$, the set
$D_\Lambda$ (see Notation~\ref{not:1}) is contained in $\{0\} \cup \left[\frac{1}{2},1\right]$. The proof is elementary: an element of $D_\Lambda$ is of the form
\begin{equation}
\label{eq:form_of_coeff}
1-\frac{1}{m} + \frac{\sum_{i=1}^k m_i\lambda_i}{m} = \begin{cases} \sum_{i=1}^k m_i\lambda_i \geq \frac{1}{2} \quad \text{ or }=0\qquad & \text{ if } m=1,\vspace{0.1 cm}\\ 
\geq 1-\frac{1}{m} \geq \frac{1}{2} & \text{ if } m>1,
\end{cases}
\end{equation}
where $m,m_i, k \in \zz_{>0}$ and $\lambda_i \in \Lambda \cup \{0, 1\}$.

Let $r\in \left[\frac{1}{2},1\right]$
be a coefficient of $B$. By Lemma~\ref{lem:producing-lcy-p1},
there exists a log Calabi--Yau pair $(\pp^1,B_{\pp^1})$ satisfying the following conditions:
\begin{itemize}
    \item the coefficients of $B_{\pp^1}$ belong to $D_\Lambda$;
    \item there is $q\in \pp^1$ such that
    ${\rm coeff}_q(B_{\pp^1})=1$; and 
    \item there is $p\in \pp^1$, with $p\neq q$, for which 
    ${\rm coeff}_p(B_{\pp^1})\in D_\Lambda(r)$.
\end{itemize}
By~\eqref{eq:form_of_coeff}, we know that ${\rm coeff}_x(B_{\pp^1})\geq \frac{1}{2}$ for any $x \in \Supp(B_{\pr 1.})$.
Then, we deduce that 
\begin{equation} 
\label{eq:deg-bp}
2=
{\rm deg}(B_{\pp^1}) = 
{\rm coeff}_p(B_{\pp^1}) +
{\rm coeff}_q(B_{\pp^1}) + 
\sum_{s\not\in \{p,q\}}{\rm coeff}_s(B_{\pp^1}) \geq \frac{3}{2} +
\sum_{s\not\in \{p,q\}}{\rm coeff}_s(B_{\pp^1}).
\end{equation} 

First, assume that ${\rm coeff}_s(B_{\pp^1})>0$ for some 
$s\not\in\{p,q\}$.
Due to~\eqref{eq:form_of_coeff} and ~\eqref{eq:deg-bp}, 
there must be a single such $s$, and it appears in $B_{\pp^1}$ with coefficient $\frac{1}{2}$, so that
\[
 {\rm coeff}_p(B_{\pp^1})=
1-\frac{1}{m_p} + \frac{m_0r+\sum_{i=1}^km_i\lambda_i}{m_p}=\frac{1}{2}.
\]
The equality implies $m_p=k=m_1=m_0=1$, $\lambda_1=0$ and $r = \frac{1}{2}$. 

Now, assume that $B_{\pp^1}$ is supported
only on $p$ and $q$.
By~\eqref{eq:deg-bp}, we write 
\[
{\rm coeff}_p(B_{\pp^1})=
1-\frac{1}{m_p} + \frac{m_0r+\sum_{i=1}^km_i\lambda_i}{m_p}=1.
\]
So we have that $m_p=k=m_1=m_0=1$, and either $\lambda_1=\frac{1}{2}$ and $ r=\frac{1}{2}$ or $\lambda_1=0$ and $ r=1$. In all cases, we conclude that 
$r\in \left\{\frac{1}{2},1\right\}$.
This means that $2B$ is an integral divisor.

Finally, if $\Lambda \subset \left( \frac{1}{2},1\right]$, i.e., $r \neq \frac{1}{2}$, the previous argument forces $r=1$, i.e., $B$ is reduced. 
\end{proof}

\subsection{Index and residues}\label{sec:residue}
In this subsection, we present an alternative proof of Proposition~\ref{cor:reducedCY} that does not use the language of dual complexes.
This alternative approach was kindly suggested to us by J\'anos Koll\'{a}r. 

Let $(X,B)$ be a proper dlt Calabi--Yau pair of coregularity $0$ and $\{ Z_i\}_{i}$ be the set of its 0-dimensional log canonical centers.
Let $r$ be a positive integer such that $r(K_{X} + B)$ is Cartier.
In particular, we have $\mathcal{O}_X(r(K_X+B))\simeq \omega_{X}^{[r]}(rB) \simeq (\omega_{X}(B))^{\otimes r}$.
Then, the Poincar\'{e} residue maps (see \cite[4.18]{Kol13})
\[
\mathcal{R}^{r}_{X \to Z_i} \colon (\omega_{X}(B))^{\otimes r} \to 
\omega_{Z_i}^{\otimes r}
\]
induce a $\kk$-linear residue map in cohomology
\begin{equation}\label{eq:residue}
\mathcal{R}^{r} \coloneqq  \big( \ldots,  \mathcal{R}^{r}_{X \to Z_i, *}, \ldots\big) \colon  H^0(X, \mathcal{O}_{X}(r( K_{X}+B)))
\to \bigoplus_{i} H^0(Z_i, \omega_{Z_i}^{\otimes r})\simeq \bigoplus_{i} \kk.
\end{equation}
The maps are defined up to sign for $r$ odd, and they are unique for $r$ even; see again \cite[4.18]{Kol13}. 

\begin{lemma}\label{lem:residueequivariant}
Let $(X,B)$ be a proper dlt pair of coregularity $0$, and $G$ be a subgroup of $\Aut(X,B)$, i.e., a subgroup of the group of automorphisms $g \colon X \to X$ such that $g^{*}(B)=B$.
For any even integer $r$, the map $\mathcal{R}^{r}$ is $G$-equivariant with respect to the natural action induced on its domain and codomain. 
\end{lemma}
\begin{proof}
Locally around a 0-dimensional log canonical center $Z_i$, a section $s$ of $H^0(X, \mathcal{O}_{X}(r( K_{X}+B)))$ is given by 
\[
a \prod_{j} \bigg(\frac{dx_j}{x_j}\bigg)^{\otimes r},
\]
where $a \in \mathcal{O}_{{X}, Z_i}$ and $\prod_j x_j$ is a local equation for ${B}$ at $Z_{i}$.
For any $g \in G$, the function $y_j \coloneqq x_i \circ g$ is a local equation for an irreducible component of $B$ passing through $g^{-1}(Z_i)$.
We obtain
\begin{equation}\label{eq:evenequivari}
\begin{split}
    g^*(\mathcal{R}^{r}_{X \to Z_i, *}(s))&=a(Z_i)\\ &=a \circ g(g^{-1}(Z_i))\\ &= \mathcal{R}^{r}_{X \to g^{-1}(Z_i), *}\left(a \circ g \prod_{j} \bigg(\frac{dy_j}{y_j}\bigg)^{\otimes r}\right)\\ &= \mathcal{R}^{r}_{X \to g^{-1}(Z_i), *}(g^*s),
\end{split}
\end{equation}
i.e., $\mathcal{R}^{r}$ is $G$-equivariant.\footnote{Note that if $r$ is odd, the equality \eqref{eq:evenequivari} holds in general only up to sign.}
\end{proof}

\begin{lemma}\label{computationresidue}
For any 0-dimensional log canonical centers $Z_1$ and $Z_2$ of the proper dlt Calabi--Yau pair $(X, B)$, we have $\mathcal{R}^{r}_{X \to Z_1, *} = \pm \mathcal{R}^{r}_{X \to Z_2, *}$.
In particular, if $r$ is even, we have $\mathcal{R}^{r}_{X \to Z_1, *} =  \mathcal{R}^{r}_{X \to Z_2, *}$.
\end{lemma}
\begin{proof}[Proof]
Without loss of generality, we can assume that $Z_1$ and $Z_2$ belong to the same 1-dimensional log canonical center $W \simeq \pp^1$.
Indeed, any two minimal log canonical centers of a log Calabi--Yau pair are $\pp^1$-linked; see \cite[Theorem 4.40]{Kol13}.
In our setup, this means that there exists a sequence of 0-dimensional log canonical centers $Z'_1, \ldots, Z'_s$ with $Z'_1=Z_1$ and $Z'_s = Z_2$ such that $Z'_{k}$ and $Z'_{k+1}$ are contained in the same 1-dimensional log canonical center, which is necessarily isomorphic to $\pp^1$.
Now, the Poincar\'{e} residue map $\mathcal{R}^{r}_{{X} \to Z_i}$ factors as
\[
(\omega_{ {X}}( {B}))^{\otimes r} \xrightarrow{\mathcal{R}^{r}_{ {X} \to W}} (\omega_{W}(Z_1+Z_2))^{\otimes r} \xrightarrow{\mathcal{R}^{r}_{W \to Z_i}} \omega_{Z_i}^{\otimes r},
\]
where $i=1,2$.
A generator of $H^0(\pp^1, \mathcal{O}_{\pp^1}(r (K_{\pp^1} + Z_1+Z_2)))$ is $(dx/x)^{\otimes r}$, which has residue $1$ at $Z_1$ and $(-1)^r$ at $Z_2$.
\end{proof}

\begin{proof}[Alternative proof of Proposition~\ref{cor:reducedCY}]
Let $q \colon (\widetilde{X}, \widetilde{B}) \to (X, B)$ be the index one cover defined in \eqref{eq:indexone}.
Up to a dlt modification, we can suppose that $(X, B)$ is dlt, and so $(\widetilde{X}, \widetilde{B})$ is too.
Let $\{ Z_i\}_{i}$ be the 0-dimensional log canonical centers of $(\widetilde{X}, \widetilde{B})$. Since the Galois group $\zz/m\zz$ of the cover $q$ is a subgroup of $\Aut(\widetilde{X}, \widetilde{B})$, the map 
\[
\mathcal{R}^{2} = \big( \ldots,  \mathcal{R}^{2}_{\widetilde{X} \to Z_i, *}, \ldots\big) \colon  H^0(\widetilde{X}, \mathcal{O}_{\widetilde{X}}(2( K_{\widetilde{X}}+\widetilde{B})))
\to \bigoplus_{i} H^0(Z_i, \omega_{Z_i}^{\otimes r})\simeq \bigoplus_{i} \kk
\]
is $\zz/m\zz$-equivariant by Lemma~\ref{lem:residueequivariant}. The image of $\mathcal{R}^{2}$ is the line spanned by the vector $(1, \ldots, 1) \in \bigoplus_{i} \kk$ by Lemma~\ref{computationresidue}, and it is fixed by $\zz/m\zz$, since the Galois group acts on $\bigoplus_{i} \kk$ by permuting the copies of $\kk$.

Since $(\widetilde{X}, \widetilde{B})$ is log Calabi--Yau, $\mathcal{R}^{2}$ is injective.
Therefore, we obtain that $H^0(\widetilde{X}, \mathcal{O}_{\widetilde{X}}(r( K_{\widetilde{X}}+\widetilde{B})))$ is invariant under the action of $\zz/m\zz$.
Then, a trivializing section of $2( K_{\widetilde{X}}+\widetilde{B})$ descends to a section of $2(K_{X}+B)$.
In particular, this implies that $m=2$.
\end{proof}

\begin{remark}\label{rem:comparison}
The two proposed proofs of Proposition~\ref{cor:reducedCY} are essentially equivalent.
One should regard the residue map \eqref{eq:residue} as the restriction of a global orientation of the dual complex $\mathcal{D}(\widetilde{B})$ to a local orientation at a point of a maximal cell $\sigma$ in $\mathcal{D}(\widetilde{B})$
\[
H^{n}(\mathcal{D}(\widetilde{B}), \zz) \to H^{n}(\mathcal{D}(\widetilde{B}),\mathcal{D}(\widetilde{B}) \setminus \sigma,  \zz)\simeq H^{n-1}(\partial \sigma, \zz) \simeq \zz;
\]
see \cite[\S 3.3, p.\ 234]{H02}.
\end{remark}

\section{Index of semi-log canonical Calabi--Yau pairs}
In this section, we prove a version of our main theorem for semi-log canonical pairs (Corollary \ref{introthm:slc}).
This is a key ingredient for the proof of Theorem~\ref{introcor-GS}; see \S\ref{sec:applicationmirrorsymmetry}. 
For the concept of semi-dlt and semi-log canonical pair,
we refer the readers to \cite[\S5]{Kol13}.
We will use the language of
admissible and pre-admissible sections.

\begin{definition}\label{def:admissible}
Let $(X,B)$ 
be a projective semi-dlt pair, possibly disconnected, of dimension $n$.
Assume that $m(K_X+B)$ is Cartier.
Let 
$(X^\nu,B^\nu)$
be the normalization of $(X,B)$.
Let $D^\nu \subset X^\nu$ be 
the normalization of
the double locus $D\subset X$.
We write
$(D^\nu,B_{D^\nu})$ 
for the dlt pair obtained by the adjunction of $(X^\nu,B^\nu)$ to $D^\nu$.
We write 
$X_i^\nu$ for the connected components of $X^\nu$
and $(X^\nu_i,B^\nu_i)$ for the pairs obtained by restriction of $(X^\nu,B^\nu)$ to $X^\nu_i$.
We define 
{\em pre-admissible}
and 
{\em admissible}
sections in $H^0(X,\mathcal{O}_X(m(K_X+B)))$
inductively using the following rules:
\begin{enumerate}
    \item a section
    \[
    s\in 
    H^0(X,\mathcal{O}_X(m(K_X+B))
    \] 
    is {\em pre-admissible} if its restriction
    $s|_{D^\nu} 
    \in H^0(D^\nu,\mathcal{O}_{D^\nu}(m(K_{D^\nu}+B_{D^\nu})))$ is admissible.
    The set of pre-admissible sections is denoted by
    $PA(X,m(K_X+B))$;
    \item a section
    \[
    s\in 
    H^0(X,\mathcal{O}_X(m(K_X+B))
    \]
    is {\em admissible} if 
    it is pre-admissible
    and 
    for every crepant birational map
    $g\colon (X^\nu_i,B^\nu_i)\dashrightarrow (X^\nu_j,B^\nu_j)$, we have that
    $g^*(s|_{X^\nu_i})=s|_{X^\nu_j}$.
    The set of admissible sections
    is denoted by
    $A(X,m(K_X+B))$.
\end{enumerate}
\end{definition}

In what follows, we recall several lemmas regarding pre-admissible and admissible sections.
The following lemma is clear by definition; see ~\cite[Remark 5.2]{Gon13}.

\begin{lemma}\label{lem:pull-back-admissible}
Let $(X,B)$ be a projective semi-dlt pair.
Let 
$\pi\colon (X^\nu,B^\nu)\rightarrow (X,B)$
be its normalization.
Assume $m(K_X+B)\sim 0$.
A section
$s\in H^0(X,\mathcal{O}_X(m(K_X+B)))$
is pre-admissible 
(resp. admissible)
if and only if
$\pi^*s\in H^0(X^\nu,\mathcal{O}_{X^\nu}(m(K_{X^\nu}+B^\nu)))$ is pre-admissible
(resp. admissible).
\end{lemma}

Pre-admissible
sections of normalizations are designed to descend to non-normal varieties, as explained in the following statement; see, e.g.,~\cite[Lemma 4.2]{Fuj00}. 

\begin{lemma}\label{lem:descending-nonormal}
Let $(X,B)$ be a projective semi-log canonical pair
for which $m(K_X+B)$ is Weil.
Let $(X^\nu,B^\nu)\rightarrow (X,B)$ be its normalization
and $(Y,B_Y)$ a dlt modification
of $(X^\nu,B^\nu)$.
Assume $m(K_Y+B_Y)$ is Cartier.
Then, 
$s\in PA(Y,m(K_Y+B_Y))$
descends to
$H^0(X,\mathcal{O}_X(m(K_X+B)))$.
\end{lemma}

In the context of connected  dlt pairs, which are not klt, 
the set of admissible sections
is the same as the set of pre-admissible sections; see, e.g.,~\cite[Proposition 3.2.15]{Xu20}.

\begin{lemma}\label{lem:PA=A}
Consider $(X,B)$ a connected projective dlt pair.
Assume that $(X,B)$ is not klt.
Assume that $m(K_X+B)\sim 0$
where $m$ is even.
Then, we have that
\[
PA(X,m(K_X+B)) = A(X,m(K_X+B)).
\] 
\end{lemma}

On the other hand, in the dlt setting, we can lift
admissible sections from the boundary
to pre-admissible sections
on the whole pair; see, e.g.,~\cite[Lemma 3.2.14]{Xu20}.

\begin{lemma}\label{lem:lifting-A-to-PA}
Assume that $(X,B)$ is a possibly disconnected,
projective dlt pair.
Assume that $m(K_X+B)\sim 0$
for some even integer $m$.
Assume that
\[
s \in A(B^{=1}, m(K_X+B)|_{B^{=1}}).
\] 
Then, there exists 
\[
t \in PA(X,m(K_X+B)) 
\] 
such that $t|_{B^{=1}}=s$.
\end{lemma}

Finally,
the following lemma
allows us to produce
admissible sections
on possibly disconnected dlt pairs, 
once we know the existence
of admissible sections on connected dlt pairs.
For the proof, see~\cite[Proposition 3.2.8]{Xu20}.

\begin{lemma}\label{lem:from-conn-to-disc}
Let $(X,B)$ be a, possibly disconnected, 
projective dlt log Calabi--Yau pair.
Assume that for every
component
$(X_i,B_i)$
of $(X,B)$,
we have a non-trivial section in
$A(X_i,m(K_{X_i}+B_i))$.
Then, we have that
$A(X,m(K_X+B))$
contains a non-trivial section.
\end{lemma}

Now, we turn to prove the main theorem of this section.

\begin{theorem}\label{thm:lcy-coreg-0-admissible}
Let $(X,B)$ be a, possibly disconnected, projective dlt log Calabi--Yau pair whose connected components all have
coregularity $0$ and Weil index $\lambda$.
Then, we have that 
$A(X,\lambda'(K_X+B))$ admits a non-trivial section,
where $\lambda'={\rm lcm}(\lambda,2)$.
\end{theorem}

\begin{proof}[Proof]
We proceed by induction on the dimension.
The statement is trivial if $\dim X = 0$.
Assume that the statement holds in dimension $n-1$.
By Theorem~\ref{introthm:main-thm}, we have that
$\lambda'(K_X+B)\sim 0$. 
Denote by $B_i$ the (normal) irreducible components of $B^{=1}$. 
The pairs $(B_i, \Diff_{B_i} (B-B_i))$, obtained by adjunction, are projective dlt log Calabi--Yau of dimension $n-1$, coregularity $0$ and Weil index that divides $\lambda'$; see \cite[Theorem 4.40]{Kol13} and Lemma~\ref{lem:weil-index-under-adjunction}.
Their disjoint union
\[\Delta \coloneqq \bigsqcup_i \, B_i \xrightarrow{\nu} B^{=1} = \bigcup_i B_i\]
is the normalization of $B^{=1}$.
Consider the pair
\[
(\Delta, B_{\Delta})\coloneqq \bigsqcup_i \, (B_i, \Diff_{B_i} (B-B_i)).
\]
By induction on the dimension, there exist non-trivial sections
\[
0\neq s_{B_i} \in 
A(B_i,\lambda'(K_{B_i}+\Diff_{B_i}(B-B_i))).
\] 
Then, by Lemma~\ref{lem:from-conn-to-disc}, there exists a section
\[
0\neq s_\Delta \in 
A(\Delta,\lambda'(K_\Delta+B_\Delta)),
\]
and by Lemma~\ref{lem:descending-nonormal} it descends to
\[s_{B^{=1}} \in H^0(B^{=1},\lambda'(K_{X}+B)|_{B^{=1}}).\] Since  $\nu^* s_{B^{=1}}=s_{\Delta}$, and by Lemma~\ref{lem:pull-back-admissible},
we actually have \[s_{B^{=1}} \in A(B^{=1},\lambda'(K_{X}+B)|_{B^{=1}}).\]
Finally, by Lemma~\ref{lem:lifting-A-to-PA} and Lemma~\ref{lem:PA=A}, 
there exists an admissible section
\[
0\neq s_{X} \in PA(X,\lambda'(K_X+B_X))=A(X,\lambda'(K_X+B)).
\] 
This finishes the induction step.
\end{proof}

\begin{remark}
The coregularity of a connected semi-log canonical log Calabi--Yau pair is the coregularity of a (or any) connected component of its normalization with the induced boundary.
Note that the independence of the connected component follows easily from \cite[Theorem 4.40 and Proposition 5.12]{Kol13}.
\end{remark}

\begin{proof}[Proof of Theorem~\ref{introthm:admissible}]
By Theorem~\ref{thm:lcy-coreg-0-admissible}, we 
have a non-trivial admissible section
$s\in A(X,\lambda'(K_X+B))$, where $\lambda'={\rm lcm}(\lambda,2)$.
By definition of admissible section, it follows that
$g^*s=s$ for every $g\in {\rm Bir}(X,B)$.
\end{proof}

\begin{proof}[Proof of Corollary~\ref{introthm:slc}]
Let $X^\nu \rightarrow X$ be the normalization of $X$
and $Y\rightarrow X^\nu$ be a dlt modification.
We write
$\pi\colon Y\rightarrow X$
and $K_Y+B_Y=\pi^*(K_{X}+B)$.
Hence, the pair $(Y,B_Y)$ is dlt log Calabi--Yau of coregularity $0$
and its Weil index is $\lambda$. 
By Theorem~\ref{thm:lcy-coreg-0-admissible}, 
we conclude that $A(Y,\lambda'(K_Y+B_Y))$
admits a non-trivial section.
Then, the statement follows from Lemma~\ref{lem:descending-nonormal}.
\end{proof}

\begin{proof}[Proof of Corollary~\ref{cor:sing}]
Let $(X, B;\, x)$ be a germ as in the statement.
Since $(X, B;\, x)$ has coregularity 0, the point $x$ is a log canonical center of $(X, B;\, x)$.
If $(X, B;\, x)$ is dlt at $x$, then the claim is immediate, as $(X, B)$ has simple normal crossings at $x$ and so $\lambda(K_X + B) \sim 0$.
Thus, in the rest of the proof, we may assume that $(X, B;\, x)$ is not dlt at $x$.

Now, let $g \colon (Y, B_{Y} \coloneqq g^{-1}_*(B) + E) \to (X,B)$ be a $\qq$-factorial dlt modification of $(X, B;\, x)$ such that $g$ is an isomorphism over the simple normal crossing locus of $(X,B)$, where $E$ is the sum of the $g$-exceptional divisors; see \cite[Corollary 1.36]{Kol13}.
Since $(X, B;\, x)$ is not dlt, it follows that $E$ is non-empty.
Furthermore, since the pair $(X,B; \, x)$ is dlt away from $x$ and $g$ is an isomorphism at a general point of the log canonical centers of $(X, B)$ contained in $X \setminus \{x\}$, it follows that $g(E)=x$, and we can also assume that $E=\Supp (g^{-1}(x))$ by \cite[Theorem 2.5]{Nak21}.
In particular, we have
\[
(K_Y + B_Y)|_{E} \sim_{\qq} g^*(K_X+B)|_{E} \sim_{\qq} 0.
\]
We denote by $(E, B_E)$ the pair induced on $E$ by $(Y, B_Y)$ via adjunction.
Then, $(E,B_E)$ is a projective semi-dlt log Calabi--Yau pair whose Weil index divides $\lambda'$ by Lemma \ref{lem:weil-index-under-adjunction}.
Hence, by Corollary~\ref{introthm:slc}, we have $\lambda'(K_Y + B_Y)|_{E} \sim 0$.

Now, consider the short exact sequence
\begin{equation}\label{eq:shortexactsequenceII}
    0 \to \mathcal{O}_{Y}(\lambda'(K_Y + B_Y) - E) \to \mathcal{O}_{Y}(\lambda'(K_Y + B_Y)) \to \mathcal{O}_{E}(\lambda'(K_Y + B_Y)) \to 0.
\end{equation}
Note that the pair $(Y, g^{-1}_*B)$ is dlt and $g|_{Z} \colon Z \to f(Z)$ is birational for every log canonical center of $(Y, g^{-1}_*B)$.
This can be showed as follows.
Assume that $Z$ is a log canonical center of $(Y, g^{-1}_*B)$ such that $g|_{Z} \colon Z \to f(Z)$ is not birational, and denote by $d$ its codimension in $Y$.
Since $g$ is an isomorphism at a general point of the log canonical centers of $(X, B)$ contained in $X \setminus \{x\}$, we must have $g(Z)=\{x\}$.
Since $(Y, g^{-1}_*B)$ is dlt and $Z$ has codimension $d$, $Z$ is a connected component of the intersection of $d$ prime divisors in $(g^{-1}_*B)^{=1}$.
Since it is also contained in some irreducible components of $E$, we get the sought contradiction, as the pair $(Y, g^{-1}_*(B) + E)$ is dlt.

Since we have
\[
-E \sim_{\qq} \lambda'(K_Y + B_Y) - E \sim_{\qq} K_{Y} + g^{-1}_*B,
\]
the vanishing result \cite[Corollary 10.38]{Kol13} gives $R^1 g_* \mathcal{O}_{Y}(\lambda'(K_Y + B_Y) - E) = 0$.
Therefore, by pushing forward \eqref{eq:shortexactsequenceII} via $g$, we obtain
\[
\mathcal{O}_{X}(\lambda'(K_X + B)) \simeq g_*\mathcal{O}_{Y}(\lambda'(K_Y + B_Y)) \to \mathcal{O}_{\{x\}}\to 0.
\]
This gives a section of $\mathcal{O}_{Y}(\lambda'(K_X + B))$ that does not vanish at $x$, i.e., $\mathcal{O}_{X}(\lambda'(K_X + B)) \sim 0$.
\end{proof}

\section{Applications to mirror symmetry}\label{sec:applicationmirrorsymmetry}

We now study the index of Calabi--Yau varieties that appear in mirror symmetry.
First, we introduce the notion of a minimal family of Calabi--Yau varieties of coregularity $0$ over a point.
This notion is modelled on the \emph{large complex limit} or \emph{maximal degeneration} condition; see also Remark~\ref{maxdeg}.

\begin{definition}\label{defn:minimalmaximal}
Let $(C, c_0)$ be a pointed smooth curve. Let $\pi \colon \mathcal{X} \to C$ be a flat projective morphism with $\pi_* \O \mathcal{X}. =\O C.$, and $\mathcal{B}$ an effective divisor such that every irreducible component of $\Supp(\mathcal{B})$ dominates $C$.
Then $\pi \colon (\mathcal{X}, \mathcal{B}) \to C$
is a \emph{minimal family of log Calabi--Yau pairs} if $K_{\mathcal{X}} + \mathcal{B} + \mathcal{X}_{c_0, \mathrm{red}} \sim_{\qq,C} 0$, where $\mathcal{X}_{c_0, \mathrm{red}}$ is the reduced fiber over $c_0$.
If $\mathcal{B}=0$, we say it is a \emph{minimal family of Calabi--Yau varieties}.

Further, we say that the family has \emph{coregularity $0$ over $c_0 \in C$} if
\begin{itemize}
     \item the pair $(\mathcal{X}, \mathcal{B}+ \mathcal{X}_{c_0, \mathrm{red}})$ is log canonical; and
     \item
     the restriction $(\mathcal{X}|_{V}, \mathcal{B}|_{V}+ \mathcal{X}_{c_0, \mathrm{red}})$
     has coregularity $0$ for any neighborhood $V \subseteq C$ of $c_0$.
\end{itemize}
\end{definition}

\begin{remark}[Maximal degeneration] \label{maxdeg}Let $\pi \colon \mathcal{X} \to C$ be a minimal family of Calabi--Yau varieties of coregularity over $c_0$.
Let $K$ denote the fraction field of the formal ring  $\widehat{\mathcal{O}}_{C, c_0}$.
If the $K$-variety $X \coloneqq \mathcal{X} \times_C \Spec K$ is smooth with trivial canonical bundle, Definition \ref{defn:minimalmaximal} recovers the notion of maximally unipotent degeneration in \cite[4.2.4.(4)]{NX16}; see also \cite[\S6.1 and \S6.2]{KLSV2018}.
Our weaker definition allows mild singularities on the general fiber, and we do not prescribe the index of the general fiber.
\end{remark}

\begin{definition}[Locally stable family]\label{def:locallystable}
Let $C$ be a smooth curve.
A family $\pi \colon (\mathcal{X},\mathcal{B}) \rar C$ is \emph{locally stable} if the pair $(\mathcal{X}, \mathcal{B} + \mathcal{X}_c)$ is semi-log canonical for any closed point $c \in C$.
\end{definition}

In the following lemma we show that, up to base change, we can suppose that the fibers of a minimal family of log Calabi--Yau pairs are all reduced and semi-log canonical.

\begin{lemma} \label{GS_lemma}
Let $\pi \colon (\mathcal{X},\mathcal{B}) \rar C$ be a minimal family of log Calabi--Yau pairs of coregularity $0$ over $c_0 \in C$.
There exists a finite morphism $(C',c'_0) \to (C,c_0)$ and a birational transformation $\mathcal{Y} \to \mathcal{X} \times_{C} C'$ such that $\pi' \colon (\mathcal{Y},\mathcal{D}) \to C'$ is a minimal family of log Calabi--Yau pairs of coregularity $0$ over $c'_0 \in C'$ with the additional property that $\pi' \colon (\mathcal{Y},\mathcal{D})\rar C'$ is locally stable at $c_0'$.

Let $U$ be the maximal open subset over which $\pi$ is locally stable.
Then, for every $c \in U$, there exists a $c' \in C'$ such that $(\mathcal{Y}_{c'},\mathcal{D}_{c'}) \simeq (\mathcal{X}_c,\mathcal{B}_c)$.
\end{lemma}

\begin{remark}
The divisor $\mathcal{D}$ in the statement of Lemma \ref{GS_lemma} is defined as follows.
Over $C' \setminus \{c'_0\}$, $\mathcal{D} \times_{C'} C' \setminus \{c_0'\}$ is the boundary obtained by crepant pull-back to $\mathcal{Y} \times_{C'} C' \setminus \{c'_0\}$ of the pair $(\mathcal{X},\mathcal{B})\times_C C \setminus \{c_0\}$.
Then, $\mathcal{D}$ is the divisor obtained by taking the closure of $\mathcal{D} \times_{C'} C' \setminus \{c_0'\}$ in $\mathcal{Y}$.
\end{remark}

\begin{proof} 
Let $U$ be the maximal open subset of $C$ where $\pi$ is locally stable.
Without loss of generality, we may assume that $U \cup \{c_0\}=C$.
Then, by weak semi-stable reduction \cite{AK00}, there is a finite morphism $(C',c'_0) \to (C,c_0)$ such that the fiber $\mathcal{Y}_{c'_0}$ over $c'_0$ of $\mathcal{Y}$ is reduced.
Here, $\mathcal{Y}$ denotes the main component of the normalization of $\mathcal{X} \times_{C} C'$. By assumption, we have $\K \mathcal{X}. + \mathcal{B}+ \mathcal{X}_{c_0,\mathrm{red}} \sim_{\qq,C}0$, and so by base change and the ramification formula, we obtain $\K \mathcal{Y}. +\mathcal{D}+ \mathcal{Y}_{c'_0} \sim_{\qq,C'}0$.
Furthermore, by construction, the family $\pi'\colon (\mathcal{Y}, \mathcal{D}) \to C'$ is locally stable.

Now, we show that $(\mathcal{Y},\mathcal{D}+\mathcal{Y}_{c'_0})$ has coregularity 0 over $c'_0$.
Up to a dlt modification, we can suppose that $(\mathcal{X},\mathcal{D} + \mathcal{X}_{c_0, \mathrm{red}})$ is dlt.
In a neighborhood of a log canonical center of dimension 0, the pair $(\mathcal{Y},\mathcal{D}'+ \mathcal{Y}_{c'_0})$ is not necessarily dlt, but it is \emph{qdlt} (see \cite[\S5]{dFKX17}) and still of coregularity 0 over $c'_0$ as we wanted.
See also \cite[Lemma 4.1.9]{NX16}.
\end{proof}

\begin{theorem}\label{introcor-GS-pairs}
Let $\pi \colon (\mathcal{X},\mathcal{B}) \rar C$ be a minimal family of Calabi--Yau varieties of coregularity $0$ over $c_0 \in C$.
Assume that $\lambda \mathcal{B}$ is integral, where $\lambda$ is a positive integer.
Let $U$ be the maximal open subset over which the family is locally stable.
Then, for every $c\in U$, we have that
$\lambda' (K_{\mathcal{X}_c}+\mathcal{B}_c)\sim 0$, where $\lambda'=\lcm(\lambda,2)$.
\end{theorem}

\begin{proof}
By shrinking $C$, we may assume that the open subset in the statement satisfies $C \setminus \{c_0\} \subset U$.
Then, by Lemma \ref{GS_lemma}, we can suppose that $\pi$ is locally stable, i.e., $U=C$.
In particular, $\K \mathcal{X}.+\mathcal{B}$ is $\qq$-Cartier.
By Corollary~\ref{introthm:slc}, we have $\lambda'(\K \mathcal{X}_{c_0}.+\mathcal{B}_{c_0}) \sim 0$.
In particular, $\lambda'(\K \mathcal{X}_{c_0}.+\mathcal{B}_{c_0})$ is Cartier.
By \cite[Proposition 2.79]{Kollar21}, we have $\mathcal{O}_{\mathcal{X}}(\lambda'(\K \mathcal{X}/C.+ \mathcal{B}))|_{\mathcal{X}_{c_0}} \simeq \mathcal{O}_{\mathcal{X}_{c_0}}(\lambda'(\K \mathcal{X}_{c_0}.+ \mathcal{B}_{c_0}))$.
Since $\mathcal{X}_{c_0}$ is a Cartier divisor in $\mathcal{X}$, by Nakayama's lemma, it follows that $\lambda'( \K \mathcal{X}.+\mathcal{B})$ is Cartier in a neighborhood of $\mathcal{X}_{c_0}$.
By the Calabi--Yau condition in Definition \ref{defn:minimalmaximal}, $\mathcal{O}_{\mathcal{X}}(\pm \lambda' (\K \mathcal{X}.+\mathcal{B}))$ is a relatively semi-ample line bundle over a neighborhood of $c_0$.
Then, by \cite[Corollary 2.65]{Kollar21}, $h^0(\mathcal{X}_c,\mathcal{O}_{\mathcal{X}_c}(\lambda'(\K \mathcal{X}_c.+\mathcal{B}_c)))=1$ for $c$ in a neighborhood of $c_0$.
Henceforth, by Lemma \ref{lemma_Weil_torsion}, we conclude that $\lambda'(\K \mathcal{X}_c.+\mathcal{B}_c) \sim 0$ for $c$ in a neighborhood of $c_0$.

Now, since $\pi$ is a locally stable family and $\lambda \mathcal{B}$ is integral, $\omega_{\mathcal{X}/C}^{[\lambda']}(\lambda'\mathcal{B})$ is a flat family of divisorial sheaves, see \cite[Proposition 2.79.3 and \S 3.3]{Kollar21}.
Then, by \cite[Theorem 3.32]{Kollar21},
$h^0(\mathcal{X}_c,\mathcal{O}_{\mathcal{X}_c}(\lambda'(\K \mathcal{X}_c.+\mathcal{B}_c))) \geq 1$ for all $c \in C$.
Since $\K \mathcal{X}_c.+\mathcal{B}_c$ is torsion for all $c \in C$, we have that $h^0(\mathcal{X}_c,\mathcal{O}_{\mathcal{X}_c}(\lambda'(\K \mathcal{X}_c.+\mathcal{B}_c))) = 1$ for all $c \in C$.
By Lemma \ref{lemma_Weil_torsion}, this concludes the proof.
\end{proof}

\begin{corollary}\label{cor:DMRorientable}
Let $\pi \colon \mathcal{X} \rar C$ be a locally stable minimal family of Calabi--Yau varieties of coregularity $0$ over $c_0 \in C$. Suppose that the general fiber $\mathcal{X}_{c}$ is klt. 
If $K_{\mathcal{X}_c} \not\sim 0$, then $\mathcal{DMR}(\mathcal{X}, \mathcal{X}_{c_0})$ is not orientable.
\end{corollary}

\begin{proof} 
Up to a small $\qq$-factorial modification of $\mathcal{X}$, we can suppose that $\mathcal{X}$ has only klt $\qq$-factorial singularities.
Take a dlt modification $(\mathcal{Y}, \mathcal{Y}_{c_0, \mathrm{red}})$ of $(\mathcal{X}, \mathcal{X}_{c_0})$ as constructed in \cite[\S 1.35]{Kol13}.
In particular, observe that $\mathcal{Y} \rar \mathcal{X}$ is an isomorphism over the open set $\mathcal{X} \setminus \mathcal{X}_{c_0}$ as $\mathcal{X}$ has only klt $\qq$-factorial singularities.

By definition of dual complex of a log Calabi--Yau pair (cf Definition \ref{rmk:dualcomplex}), we write
\[\mathcal{DMR}(\mathcal{X}, \mathcal{X}_{c_0}) \simeq_{\mathrm{PL}} \mathcal{D}(\mathcal{Y}_{c_0, \mathrm{red}}).\]
Since $\pi \colon \mathcal{X} \to C$ has coregularity 0 over $c_0$, there exists an isomorphism
\[H^{n}(\mathcal{D}(\mathcal{Y}_{c_0, \mathrm{red}}), \kk) \simeq H^{n}(\mathcal{Y}_{c_0, \mathrm{red}}, \mathcal{O}_{\mathcal{Y}_{c_0, \mathrm{red}}})\]
with $n=\dim \mathcal{Y}_{c_0, \mathrm{red}}$, see \cite[Lemma 25]{KX16} and \cite[Proposition A.7]{KLSV2018}.
Recall that the fibers of $\pi$ have du Bois singularities, and $\mathcal{Y} \rar \mathcal{X}$ is an isomorphism over the open set $\mathcal{X} \setminus \mathcal{X}_{c_0}$. Then, \cite[Theorem 4.3]{Schwede2007} gives the following isomorphism:
\[
H^{n}(\mathcal{Y}_{c_0, \mathrm{red}}, \mathcal{O}_{\mathcal{Y}_{c_0, \mathrm{red}}}) \simeq H^{n}(\mathcal{X}_{c_0}, \mathcal{O}_{\mathcal{X}_{c_0}}).
\]
By \cite[Theorem 2.62]{Kollar21}, we also have
 \[ H^{n}(\mathcal{X}_{c_0}, \mathcal{O}_{\mathcal{X}_{c_0}}) \simeq H^{n}(\mathcal{X}_{c}, \mathcal{O}_{\mathcal{X}_{c}}).\]
Now, since the general fiber $\mathcal{X}_{c}$ is Cohen--Macaulay, 
by duality we obtain
\[H^{n}(\mathcal{X}_{c}, \mathcal{O}_{\mathcal{X}_{c}}) \simeq H^{0}(\mathcal{X}_{c}, \mathcal{O}_{\mathcal{X}_{c}}(K_{\mathcal{X}_{c}})).\]
Finally, we conclude
\[H^n(\mathcal{DMR}(\mathcal{X}, \mathcal{X}_{c_0})) \simeq H^{0}(\mathcal{X}_{c}, \mathcal{O}_{\mathcal{X}_{c}}(K_{\mathcal{X}_{c}})).\]
If $K_{\mathcal{X}_{c}} \not \sim 0$, then $H^{n}(\mathcal{DMR}(\mathcal{X}, \mathcal{X}_{c_0}), \kk)=0$, i.e., $\mathcal{DMR}(\mathcal{X}, \mathcal{X}_{c_0})$ is not orientable. 
\end{proof}

\begin{proof}[Proof of Corollary \ref{cor:nonorientamirror}]
Corollary \ref{cor:nonorientamirror} is a special case of  Corollary \ref{cor:DMRorientable}.
\end{proof}
See Example \ref{ex:quotientfamily} for an instance of Corollary \ref{cor:nonorientamirror}.

\begin{proof}[Proof of Theorem~\ref{introcor-GS2}] 
Up to shrinking $C$, we can assume that $\pi$ is locally stable. 
By the same proof of Theorem~\ref{introcor-GS}, it suffices to show that $\K \mathcal{X}_{c_0}. \sim 0 $, or $ \K \mathcal{X}. + \mathcal{X}_{c_0} \sim 0$ in a neighborhood of $\mathcal{X}_{c_0}$. 

First, we show that
the restriction of
$\K \mathcal{X}. + \mathcal{X}_{c_0}$
to each component $\Delta_i$ of the central fiber
is linearly trivial.
To this end, note that the $\qq$-divisor
\[
\K \Delta_i. + \Diff_{\Delta_i}(\mathcal{X}_{c_0} - \Delta_i) \sim_{\qq} (\K \mathcal{X}. + \mathcal{X}_{c_0})|_{\Delta_i} \sim_{\qq} 0
\]
is actually a linearly trivial Cartier divisor, since the inequality
\[
0 \sim_{\qq} \K \Delta_i. + \Diff_{\Delta_i}(\mathcal{X}_{c_0}-\Delta_i) \geq \K \Delta_i. + \sum_{j \neq i} \Delta_j|_{\Delta_i} \sim 0
\]
implies that $\K \Delta_i. + \Diff_{\Delta_i}(\mathcal{X}_{c_0}-\Delta_i) = \K \Delta_i. + \sum_{j \neq i} \Delta_j$, and $\K \Delta_i. + \sum_{j \neq i} \Delta_j \sim 0$ as for any toric pair.

Let $m$ be the smallest positive integer for which
$m(K_{\mathcal{X}}+\mathcal{X}_{c_0})\sim 0$
holds, up to eventually shrinking $C$.
Take the degree $m$  quasi-\'{e}tale
cover $q \colon \mathcal{X}'\rightarrow \mathcal{X}$ associated to $m(K_{\mathcal{X}}+\mathcal{X}_{c_0})\sim 0$. 
By \cite[(4.2.7)]{Kol13}, there exists an open subset $\Delta^{\circ}_{i} \subseteq \Delta_i$, with $\codim_{\Delta_i}(\Delta_i \setminus \Delta^{\circ}_{i})\geq 2$, such that for any $k>0$ \[\omega^{[k]}_{\mathcal{X}}(k\mathcal{X}_{c_0})|_{\Delta^{\circ}_i} \simeq \omega^{[k]}_{\mathcal{X}_{c_0}}(k\Diff_{\Delta_i}(\mathcal{X}_{c_0} -\Delta_i))|_{\Delta^{\circ}_i}\sim 0.\] This implies that $q^{-1}(\Delta^{\circ}_i)$
consists of $m$ disjoint copies of $\Delta^{\circ}_i$.
By the smoothness of $\Delta_i$ and purity of the branch locus, or by \cite[Lemma 2.93]{Kollar21}, the same holds for $q^{-1}(\Delta_i)$.
Therefore, the induced morphism
$
\mathcal{X}'_{c_0} \rightarrow \mathcal{X}_{c_0}
$ 
is a finite \'etale cover of degree $m$. 
Since every irreducible component $\Delta_i$
of $\mathcal{X}_{c_0}$ is simply connected, 
it follows that
\[
\pi^{\text{ét}}_1(\mathcal{X}_{c_0}) \simeq 
\pi_1(\mathcal{D}(\mathcal{X}_{c_0})) \simeq \{1\}
\] 
by \cite[Lemma 26]{KX16} and by the hypothesis on the simple connectedness of $\mathcal{D}(\mathcal{X}_{c_0})$.
This implies that $m=1$ and finishes the proof.
\end{proof}

\begin{remark}
In Theorem~\ref{introcor-GS2}, we require that the special fiber $\mathcal{X}_{c_0} = \sum \Delta_i$ is toric, but we could weaken it by requiring that all the log Calabi--Yau pairs $(\Delta_i, \Diff_{\Delta_i}(\mathcal{X}_{c_0} - \Delta_i))$ have index 1. By Theorem~\ref{introcor-GS}, the cover $q \colon \mathcal{X}'\rightarrow \mathcal{X}$ in the proof of Theorem~\ref{introcor-GS2} has actually degree at most 2.
Therefore, one could also replace the assumption on the simple connectedness of the dual complex with the weaker condition that $\pi_1(\mathcal{D}(\mathcal{X}_{c_0}))$ has no index $2$ subgroups.
\end{remark}

\section{Quotients of Calabi--Yau and holomorphic symplectic varieties}\label{sec:quotients}

As a consequence of Theorem \ref{introcor-GS}, we obtain the following bound on the index of quotients of log Calabi--Yau varieties.

\begin{theorem}\label{thm_group_actions_general}
Let
\begin{itemize}
    \item $(X, \Delta_X)$ be a projective log Calabi--Yau pair with $K_X + \Delta_X \sim_\qq 0$ where $\Delta_X = \sum_i (1 - \frac{1}{m_i})\Delta_i$ has standard coefficients;
    \item $G$ be a finite subggroup of $\mathrm{Aut}(X, \Delta_X)$, i.e., the group of automorphisms $g \colon X \to X$ such that $g^*\Delta_X=\Delta_X$; and
    \item $(Y, \Delta_Y)$ be the pair given by the quotient $Y \coloneqq X/G$ and the boundary $\Delta_Y$ induced from $(X, \Delta_X)$ by the Riemann--Hurwitz formula as in \cite[(2.41.6)]{Kol13}.
\end{itemize}
Assume that $(Y, \Delta_Y)$ is a fiber of a minimal family $\pi \colon (\mathcal{Y},\mathcal{B}) \rar C$ of Calabi--Yau pairs of coregularity $0$ over $c_0 \in C$ such that $\coeff(\mathcal{B})=\coeff(\Delta_Y)$.
Then
the index of $(Y, \Delta_Y)$ is at most $2$, or equivalently the order of the character $\rho \colon G \to \mathrm{GL}(H^0(X,\O X.(K_X + \Delta_X)))$ is at most $2$.
\end{theorem}

Note that the condition $\coeff(\mathcal{B})=\coeff(\Delta_Y)$ is satisfied by a general fiber of the family $\pi$.

\begin{proof} Assume that $(Y, \Delta_Y)$ admits a degeneration $\pi \colon (\mathcal{Y},\mathcal{B}) \rar C$ as in the statement.
Up to a finite base change ramified at $c_0$, by Proposition \ref{prop_quotient_coreg}, we may further assume that $\mathcal{Y}_{c_0}$ is reduced and hence $\pi \colon (\mathcal{Y},\mathcal{B}) \rar C$ is a locally stable family near $c_0$.
Then, by adjunction, $(\mathcal{Y}_{c_0},\mathcal{B}_{c_0})$ is a semi-log canonical Calabi--Yau pair of coregularity 0.

By the Riemann--Hurwitz formula, the coefficients of $\Delta_Y$ are standard, see \cite[(2.41.6)]{Kol13}.
Then, by hypothesis, the coefficients of $\mathcal{B}$ are standard as well.
Since standard coefficients are greater than or equal to $\frac{1}{2}$, components of $\mathcal{B}$ that intersect along a divisor $D_{c_0}$ of $\mathcal{Y}_{c_0}$ must have coefficient $\frac{1}{2}$, so that $D_{c_0}$ appears with coefficient 1 in $\mathcal{B}_{c_0}$.
In particular, also $\mathcal{B}_{c_0}$ has standard coefficients.

Then, by applying Lemma~\ref{lem:producing-lcy-p1} to the normalization of $(\mathcal{Y}_{c_0},\mathcal{B}_{c_0})$, it follows that the coefficients of $\mathcal{B}_{c_0}$, and hence of $\mathcal{B}$, are in $\lbrace \frac{1}{2},1\rbrace$.
Therefore, by Theorem~\ref{introcor-GS-pairs}, we have that $2(K_{Y}+\Delta_Y) \sim 0$.
Since the sections of $\O Y.(2(\K Y.+\Delta_Y))$ correspond to $G$-invariant sections of $\O X.(2(\K X.+\Delta_X))$, it follows that $\rho^{2}$ is the trivial character.
\end{proof}

We now apply Theorem \ref{thm_group_actions_general} to degeneration of holomorphic symplectic varieties.

\begin{definition}\label{def:holosympl}
A normal proper variety $X$ with canonical (Gorenstein) singularities is:
\begin{enumerate}
    \item \emph{holomorphic symplectic} if it admits a nondegenerate closed holomorphic 2-form $\omega_X \in H^0(X^{\mathrm{reg}}, \Omega^2_{X^{\mathrm{reg}}})$ on its regular locus;
    \item \emph{primitive symplectic} if it is holomorphic symplectic, $H^1(X, \mathcal{O}_X)=0$ and $H^0(X, \Omega^{[2]}_X)$ is generated by $\omega_X$. 
\end{enumerate} 
\end{definition}
\begin{definition}
An automorphism $g \colon X \to X$ of a holomorphic symplectic variety $X$ is \emph{non-symplectic} if $g^*\omega_X \neq \omega_X$, and it is \emph{purely non-symplectic} if $(g^{k})^{*}\omega_X \neq \omega_X$ for all powers $g^k \neq 1$, e.g., a non-symplectic automorphism of prime order.
\end{definition}

If $g$ is finite of order $m$,\footnote{Note that if $X$ is a projective primitive symplectic, a purely non-symplectic automorphism is automatically finite by the same proof of \cite[Corollary 3.4.]{Huybrehcts2016}.} any non-symplectic automorphism descends to a purely non-symplectic automorphism on the quotient $X/g^{k_0}$, where $g^{k_0}$ is the minimal symplectic power of $g$.

Let $(C, c_0)$ denote the germ of a smooth curve at a point $c_0 \in C$ and let $C^* \coloneqq C \setminus \{c_0\}$.
Let $\pi^* \colon X^{*} \to C^*$ be a projective family of holomorphic symplectic varieties.

\begin{definition}
The degeneration $\pi^*$ is \emph{of type III} if, up to finite base change, $\pi^*$ extends to a minimal family $\pi \colon \mathcal{X} \to C$ of Calabi--Yau varieties of coregularity $0$ over $c_0$.
\end{definition}

See \cite[Theorem 0.11]{KLSV2018} for a discussion of the other types of degenerations and the relation with the coregularity of the central fiber.

\begin{proof}[Proof of Corollary \ref{thm_hyperkahler}] 
Let $\pi \colon \mathcal{X} \to C$ be an extension of $\pi^* \colon \mathcal{X}^* \to C^*$, projective over $C$. Up to a birational modification, we can suppose that $g \colon \mathcal{X}^* \to \mathcal{X}^*$ extends to a regular morphism on $g \colon \mathcal{X} \to \mathcal{X}$, e.g., take the normalization of the main component of the graph of the rational map $\mathcal{X} \dashrightarrow \prod^{m-1}_{i=1} \mathcal{X}$ sending $x \mapsto (g(x), \ldots, g^{m-1}(x))$, where $m$ is the order of $g$.
The semi-stable reduction theorem for Calabi--Yau varieties in \cite{Fujino2011} works in the equivariant setting too: it suffices to take $g$-equivariant log resolutions \cite[Proposition 3.9.1]{Kol07} and to replace the ordinary MMP with a $g$-equivariant analog. Therefore, we can suppose that $\pi \colon \mathcal{X} \to C$ is a minimal family of Calabi--Yau varieties endowed with an automorphism $g \colon \mathcal{X} \to \mathcal{X}$ acting fiberwise and such that the pair $(\mathcal{X}, \mathcal{X}_{c_0} = \mathcal{X}_{c_0, \mathrm{red}})$ is dlt.
In particular, $g^*(K_{\mathcal{X}}+\mathcal{X}_{c_0})\sim_{\qq} K_{\mathcal{X}}+\mathcal{X}_{c_0} \sim_{\qq, C}0$.

Consider the pair $(\mathcal{Y}, \Delta_{\mathcal{Y}})$ given by the quotient $\mathcal{Y} \coloneqq \mathcal{X}/g$ and the boundary $\Delta_{\mathcal{Y}}$ induced from $(\mathcal{X}, \mathcal{X}_{c_0})$ by the Riemann--Hurwitz formula as in \cite[(2.41.6)]{Kol13}. Let $\mathcal{B}$ be the sum of the components of $\Delta_{\mathcal{Y}}$ dominating $C$. By construction, we have
\[0 \sim_{\qq, C} K_{\mathcal{Y}} + \Delta_{\mathcal{Y}} = K_{\mathcal{Y}} + \mathcal{B} + \mathcal{Y}_{c_0} = K_{\mathcal{Y}} + \mathcal{B} + \mathcal{Y}_{c_0, \mathrm{red}}.\]
Hence, the quotient family $\pi_Y \colon (\mathcal{Y}, \Delta_{\mathcal{Y}}) \to C$ is a minimal family of Calabi--Yau varieties. If $\mathcal{X}^*$ is of type III, then the pair $(\mathcal{X}, \mathcal{X}_{c_0})$ has coregularity $0$, and so the pair $(\mathcal{Y}, \Delta_{\mathcal{Y}})$ has coregularity $0$ too by Proposition \ref{prop_quotient_coreg}.
Let $\omega \in H^0(\mathcal{X}_c, \Omega^{[2]}_{\mathcal{X}_c})$. Since $g$ acts purely non-symplectically on the general fiber $\mathcal{X}_c$, then $g^*\omega = \zeta_{m} \omega$, where $\zeta_{m}$ is a primitive $m$-th root of unity, and so $g^*\omega^{n} = \zeta^{n}_{m} \omega^{n} \in H^0(\mathcal{X}_c, \mathcal{O}_{\mathcal{X}_c}(K_{\mathcal{X}_c}))$. The index of the general fiber $\mathcal{Y}_c$ is equal to the order of $\zeta^{n}_{m}$.
By Theorem \ref{thm_group_actions_general}, we conclude that $m$ divides $\dim \mathcal{X}_c=2n$.
\end{proof}

\begin{remark}\label{rmk:finaldual}
The dual complex $\mathcal{DMR}(\mathcal{Y}, \Delta_{\mathcal{Y}})$ is PL-homeomorphic to the quotient $\mathcal{D}(\mathcal{X}_{c_0})/g$ by Proposition \ref{prop:dualcomplexquotient}. In particular, if $g$ is a non-symplectic involution and $\dim \mathcal{X}_{c} \equiv 2 \mod 4$, then $\mathcal{DMR}(\mathcal{Y}, \Delta_{\mathcal{Y}})$ is not orientable, since the identity \eqref{eq:identidualcomplex} is $g$-equivariant, and $g$ acts non-trivially on $H^0(\mathcal{X}_{c}, \mathcal{O}_{\mathcal{X}_{c}}(K_{\mathcal{X}_{c}}))$.
\end{remark}

\begin{remark}
Let $\mathcal{F}$ be the moduli space of (markable)
primitive symplectic varieties of dimension $n$ with a purely non-symplectic automorphism of order $2$. A consequence of Remark~\ref{rmk_alexeev} is that the Baily--Borel compactification of the period domain of $\mathcal{F}$ does not contain 0-cusps. We refer to \cite[\S 2]{AEC2021} for an explanation of the terminology. In fact, we do not know any place in the literature where the details of the construction of these moduli spaces are carried out for arbitrary primitive symplectic varieties. For that, one can adapt the arguments in \cite[\S 2]{AEC2021} for K3 surfaces and \cite[\S 4]{BCS2016}. This goes beyond the scope of this paper, so we omit the details here. 
\end{remark}

\section{Examples}

In this section, we collect some examples showing that the statement of the theorem is optimal.

\begin{example}[{\cite[Example 60]{KX16}}]\label{ex:RPN}
Set $X_1 \coloneqq \pr 1.$ and $\Delta_1 \coloneqq (0:1)+(1:0)$.
Let $\tau_1$ denote the involution $\tau_1 \colon (x:y) \mapsto (y:x)$.
Then, for every integer $n \geq 2$, we set $(X_n,\Delta_n) \coloneqq (X_1,\Delta_1)^n$, while $\tau_n$ denotes the involution $\tau_n \coloneqq (\tau_1,\ldots ,\tau_1)$ that acts diagonally as $\tau_1$ on each factor.
We define $(Y_n,B_n) \coloneqq (X_n,\Delta_n)/(\tau_n)$.\footnote{For $n=2$, the variety $Y_2$ is a so-called Fano--Enriques threefold, and in characteristic $2$, it is a Fano compactification of a terminal non-Cohen-Macaulay singularity studied in \cite[Theorem 5.1]{Totaro2019}.}

\begin{lemma}
If $n \geq 2$, $(Y_n,B_n)$ is a dlt log Calabi--Yau pair, and $B_{n}$ is a reduced Cartier divisor.
\end{lemma}
\begin{proof}
If $n \geq 2$, the action of $\tau_n$ is free in codimension 1. The morphism $X_n \rar Y_n$ is \'etale in codimension 1, so the pair $(Y_{n}, B_{n})$ is log Calabi--Yau. The quotient $X_n \rar Y_n$ is \'etale along $\Delta_n$, and $B_{n}$ is a reduced Cartier divisor with only simple normal crossings, since $\Delta_n$ is so. 
Hence, $(Y_n,B_n)$ is a dlt pair.
\end{proof}

\begin{lemma}\label{lem:importanceof2}
If $n \geq 2$ is odd, $\K Y_n. + B_n \not \sim 0$ but $2(\K Y_n. +B_n) \sim 0$.
\end{lemma}
\begin{proof}
We observe that the log canonical divisor $\K Y_n. + B_n$ is not Cartier. Indeed, the involution $\tau_{n}$ does not preserve the volume form $dz_1/z_1 \wedge dz_2/z_2 \wedge \ldots \wedge dz_n/z_n$ on $\mathbb{G}^{n}_m = X_{n} \setminus \Delta_n$, where $z_i \coloneqq x/y$ is a coordinate on the torus of the $i$-th copy of $\pp^1$ in $X_{n} = (\pp^1)^n$. However, $2(\K Y_n. +B_n) \sim 0$, since the class group of $Y_{n}$ has only 2-torsion.

Alternatively, note that the dual complex $\mathcal{D}(\Delta_n)$ is an $n$-dimensional  hyperoctahedron, thus it is PL-homeomorphic to $\mathbb{S}^{n-1}$.
The involution $\tau_{n}$ induces the antipodal involution on $\mathbb{S}^{n-1}$.
Hence, $\mathcal{D}(B_n) \simeq \mathbb{RP}^{n-1}$, which is not orientable for $n \geq 2$ odd.
By Lemma \ref{lem:pseudomflddualcomplex}, it follows that $\K Y_n. + B_n \not \sim 0$, while $2(\K Y_n. +B_n) \sim 0$ by descending $\K X_n. + \Delta_n \sim 0$. 
\end{proof}

Examples of the behavior in Lemma \ref{lem:importanceof2} in even dimension $n > 2$ can be achieved by replacing $(Y_{n}, B_{n})$ by $(Y_{n-1}, B_{n-1}) \times (X_1, \Delta_1)$; see also \cite[Example 7.4]{Mauri2020}.
This shows that the least common multiple in Theorem \ref{introthm:main-thm} and Proposition \ref{cor:reducedCY} is indeed necessary.
In view of Theorem~\ref{introcor-GS2}, it is also interesting to note that $B_n$ is a toric simple normal crossing variety but $\pi_1(B_n) \simeq \zz/2\zz$ and $K_{B_n} \not \sim 0$. 
\end{example}

\begin{example}\label{ex:quotientfamily}
With the notation of Example~\ref{ex:RPN}, let $[x_i : y_i]$ be homogeneous coordinates on the $i$-th copy of $\pp^1$ in $X_{n} = (\pp^1)^n$.
The pencil
\[\mathcal{X} \coloneqq \left\{ ([x_1:y_1],\dots,[x_n:y_n],t) \in X_n \times \mathbb{A}^1 \mid \prod_i x_i^2 + \prod_i y^2_i + t \prod_i x_i y_i =0 \right\} \to \mathbb{A}^1\]
is an example of a minimal family of Calabi--Yau varieties of coregularity 0 at $t=0$, locally stable in a neighborhood of $t=0$.
The fiber $\mathcal{X}_0$ is the toric boundary of $X_{n}$, so the assumptions of Theorem~\ref{introcor-GS2} hold for $\mathcal{X}$. Therefore, the general fiber $\mathcal{X}_{t}$ satisfies $K_{\mathcal{X}_{t}} \sim 0$, which is also clear since $\mathcal{X}_{t}$ is a smooth anticanonical hypersurface of $X_{n}$.

Note that $\mathcal{X}$ is $\tau_n$-invariant.
The quotient $\mathcal{Y} \coloneqq \mathcal{X}/\tau_n \to \mathbb{A}^1$ is still a minimal family of Calabi--Yau varieties of coregularity 0 at $t=0$, locally stable in a neighborhood of $t=0$.
The fiber $\mathcal{Y}_0$ is the boundary $B_{n}$. We are no longer in the assumptions of Theorem~\ref{introcor-GS2} as $\pi_1(B_n) \simeq \zz/2\zz$. In fact we have $K_{\mathcal{Y}_{t}} \not \sim 0$, but $2 K_{\mathcal{Y}_{t}} \sim 0$ according to Theorem~\ref{introcor-GS}. For instance, if $n=3$ and $t$ is general, $\mathcal{Y}_{t}$ is an Enriques surface.
\end{example}

\begin{example}[{\cite[Example 61]{KX16}}]\label{ex:permutation}
The cyclic group $\zz/m\zz$ acts on $\pp^{m-1}$, with $m>2$, by permuting the homogeneous coordinates $[x_1: \ldots : x_m]$.
Set $\Delta \coloneqq \{x_1 \cdot \ldots \cdot x_m=0\}$.
The quotient $(Y, B) \coloneqq (\pp^{m-1}, \Delta)/(\zz/m\zz)$ is a log Calabi--Yau pair with reduced boundary and with complicated self-intersection, and it has index $2$ if $m$ is even, and index $1$ if $m$ is odd.
Indeed, a generator of $H^0(\pp^m,\mathcal{O}_{\pp^m}( K_{\pp^m} + \Delta))$ is given by
\[\omega \coloneqq \sum_{i} (-1)^i x_i dx_0 \wedge \ldots \wedge \widehat{ d x_i} \wedge \ldots \wedge dx_n.\]
A permutation $\sigma$ of the coordinates sends $\omega$ to $\mathrm{sgn}(\sigma) \cdot \omega$, so $\sigma$ preserves $\omega^{\otimes 2}$, i.e., $2(K_{Y} + B) \sim 0$.
This shows that although the order of the group we quotient by (i.e., $m$) can be arbitrarily large, the index of the quotient log Calabi--Yau pair $(Y, B)$ is always at most 2.\footnote{A variation of Example \ref{ex:permutation} in positive characteristic appears in \cite[Proposition 3.6]{T22}.}
\end{example}

\begin{example}
Theorem \ref{introthm:main-thm} does not impose any constraints on the index of $X$, i.e., the smallest positive integer $n$ such that $n K_X$ is Cartier.
Take for instance the weighted projective space $\pp(1,1,n)$ with its toric boundary.
For $n>2$, the Cartier index of $\pp(1,1,n)$ is $n$, while the index of any toric pair is $1$. 
\end{example}

\begin{example}
\label{ex:pseudo-not-double-cover}
In general, a pseudo-manifold need not have an orientable double cover which is a topological covering space.
For instance, we can consider the suspensions $S\mathbb{S}^2$ and $S\mathbb{RP}^2$, with vertices $\{a,b\}$ and $\{p,q\}$, respectively.
Notice that the pseudo-manifold $S\mathbb{S}^2$ is orientable while $S\mathbb{RP}^2$ is not.
The space $S\mathbb{RP}^2$ is simply connected, so it does not admit any topological covering space.
On the other hand, the morphism $S\mathbb{S}^2 \rar S\mathbb{RP}^2$ provides an orientable double cover ramified at $\{p,q\}$. 
This is an incarnation of a more general fact, as every pseudo-manifold admits a \emph{branched} orientation double cover; see \cite[\S5.2]{Matthews}.
\end{example}

\bibliographystyle{habbvr}
\bibliography{bib}

\end{document}